\documentclass{amsart}

\makeatletter
\@namedef{subjclassname@2020}{%
  \textup{2020} Mathematics Subject Classification}
\makeatother 

\usepackage{amsthm,amssymb,amsfonts,latexsym,mathtools,thmtools}
\usepackage[T1]{fontenc}
\usepackage{tikz-cd} 
\usepackage{enumitem} 
\usepackage{hyperref} 
\usepackage{hyperref}
\hypersetup{
    colorlinks=true,
    linkcolor=blue,
    filecolor=blue,      
    urlcolor=cyan,
    linktocpage=true
}

\newtheorem{theorem}{Theorem}[section]
\newtheorem{lemma}[theorem]{Lemma}

\theoremstyle{definition}
\newtheorem{definition}[theorem]{Definition}

\theoremstyle{remark}
\newtheorem{remark}[theorem]{Remark}

\newtheorem{proposition}[theorem]{Proposition}
\newtheorem{corollary}[theorem]{Corollary}

\numberwithin{equation}{section}

\begin{document}

\title[Differential smoothness of 3-dimensional and diffusion algebras]{On the differential smoothness of 3-dimensional skew polynomial algebras and diffusion algebras}



\author{Armando Reyes}
\address{Universidad Nacional de Colombia - Sede Bogot\'a}
\curraddr{Campus Universitario}
\email{mareyesv@unal.edu.co}
\thanks{}


\author{Cristian Sarmiento}
\address{Universidad Nacional de Colombia - Sede Bogot\'a}
\curraddr{Campus Universitario}
\email{cdsarmientos@unal.edu.co}

\thanks{The authors were supported by the research fund of Faculty of Science, Code HERMES 52464, Universidad Nacional de Colombia - Sede Bogot\'a, Colombia.}



\date{}

\dedicatory{Dedicated to professor Oswaldo Lezama for his brilliant academic career at the \\ Universidad Nacional de Colombia - Sede Bogot\'a}

\begin{abstract}

In this paper, we study the differential smoothness of 3-dimensional skew polynomial algebras and diffusion algebras.

\end{abstract}

\maketitle


\section{Introduction}\label{introduction}


\label{intro}

Brzezi{\'n}ski \cite{Bre15} investigated the construction of an algebraic differential structure of an affine algebra $\mathcal{A}$ by defining a notion of smoothness over the algebra $\mathcal{A}$. He considered the existence of two complexes of $\mathcal{A}$-modules related through an isomorphism which is analogous to the star isomorphism of Hodge's theory, as it was obtained for ${\rm M}_{n}(\mathbb{C})$ (the set of matrices of size $n$ times $n$ with entries in the set of complex numbers $\mathbb{C}$) in the work of Dubois et al., \cite{DVKM90} (this isomorphism motivates the features of the definition of integrable differential calculus), where the dimensionality of these structures is conditioned by the Gelfand-Kirillov dimension of $\mathcal{A}$.

In this way, Brzezi{\'n}ski et al., \cite{Brz16} and \cite{BKL10}, studied the notion of twisted multiderivation $(\sigma,\partial)$, where $\sigma:\mathcal{A}\rightarrow {\rm M}_{n}(\mathcal{A})$ is an automorphism of algebras and $\partial:\mathcal{A}\rightarrow\mathcal{A}^{n}$ is a $\mathbb{K}$-linear function ($\mathbb{K}$ a field) such that $\partial(ab)=\partial(a)\sigma(b) +a\partial(b)$, for all $a, b\in\mathcal{A}$, with the purpose of introduce a first-order differential calculus $(\Omega^{1}(\mathcal{A}),d)$ (\cite{Brz16}, Definition 4.1), accompanied by a divergence function, also called a hom-connection $\nabla$, that is, a $\mathbb{K}$-linear function $\nabla:\text{Hom}_{\mathcal{A}}(\Omega^{1}(\mathcal{A}),\mathcal{A})\rightarrow\mathcal{A}$ that satisfies $\nabla(fa)=\nabla(f)a+f(d(a))$, for all $f\in\text{Hom}_{\mathcal{A}}(\Omega^{1}(\mathcal{A}),\mathcal{A})$ and $a\in\mathcal{A}$. Using a set of skew derivations relative to automorphisms of $\mathcal{A}$, Brzezi{\'n}ski et al., \cite{BKL10} built a twisted multiderivation, where $\sigma\in {\rm M}_{n}(\text{End}(\mathcal{A}))$ is a diagonal matrix (\cite{BKL10}, p. 292). This idea allowed to Brzezi{\'n}ski \cite{Brz16} to find examples of these first-order differential calculus and hom-connections for generalized Weyl algebras of degree one (introduced by Bavula \cite{Bav92}), and skew polynomial rings (defined by Ore \cite{Ore1933}) on $\mathbb{C}[x]$, see Brzezi{\'n}ski \cite{Bre15}. It is important to say that using a hom-connection $\nabla$, we obtain a notion of first-order integral calculus (Brzezi{\'n}ski \cite{Brz16}, Definition 4.6), and a resolution of $\mathcal{A}$-modules called the complex of integral forms ($\mathcal{I}\mathcal{A},\nabla$), extending $\nabla$ (Brzezi{\'n}ski \cite{BKL10}). The existence of a complex isomorphism between ($\mathcal{I}\mathcal{A},\nabla$) and $(\Omega(\mathcal{A}),d)$ determines the differential smoothness of $\mathcal{A}$ (Brzezi{\'n}ski \cite{BS14}). Therefore, it is interesting to study the conditions of the twisted derivations in order to construct the Brzezi{\'n}ski's differential calculus and then study the property of being differentially smooth, at least for affine algebras in a field.

With all the above facts in mind and following the ideas developed by Brzezi{\'n}ski, our purpose in this article is to study the differential smoothness of noncommutative algebras of polynomial type known as {\em 3-dimensional skew polynomial algebras} defined by Bell and Smith \cite{BellSmith1990}, and {\em diffusion algebras} introduced by Isaev et al., \cite{IPR01} (cf. Pyatov and Twarock \cite{PT02}). 

The article is organized as follows. Section \ref{sec:3} contains the study of the differential smoothness of 3-dimensional skew polynomial algebras. Next, Section \ref{sec:5} presents several results about the  differential smoothness of diffusion algebras. Finally, we present some ideas for a possible future work.

Throughout the paper, $\mathbb{N}$ denotes the set of natural numbers including zero, and $K$ and $\mathbb{K}$ denote a commutative ring with identity and a field, respectively. If $S$ is a set, $|S|$ denotes its number of elements. The word ring means an associative ring with identity not necessarily commutative.

\section{Differential smoothness of 3-dimensional skew polynomial algebras}
\label{sec:3}

We start establishing some preliminary definitions and results about differential calculus.

Following Gianchetta et al., \cite{GMS05}, Section 1.6, a {\em graded algebra} $\Omega^{*}$ over $K$ is defined as a direct sum $\Omega^{*} = \bigoplus_{k\in \mathbb{N}} \Omega^{k}$ of $k$-modules, provided with an associative multiplication law $a \wedge b$, $a, b \in \Omega^{*}$, such that $a \wedge b\in \Omega^{|a| + |b|}$, where $|a|$ denotes the degree of an element $a\in \Omega^{|a|}$. Note that $\Omega^{0}$ is a (noncommutative) $K$-algebra $\mathcal{A}$, while $\Omega^{k > 0}$ are $\mathcal{A}$-bimodules and $\Omega^{*}$ is an $(\mathcal{A}-\mathcal{A})$-algebra. A graded algebra $\Omega^{*}$ is said to be {\em graded commutative} if $a\wedge b = (-1)^{|a||b|} b\wedge a$, for elements $a, b \in \Omega^{*}$.

A graded algebra $\Omega^*=\bigoplus_{k\in \mathbb{N}} \Omega^k$, where $\Omega^0\cong \mathcal{A}$, is said to be a \textit{differential calculus} over $\mathcal{A}$ if it is a cochain complex of $\mathbb{K}$-modules (also known as {\em de Rham complex} of the differential graded algebra $(\Omega^{*},d)$)
\begin{equation*}
    0\rightarrow\mathbb{K}\rightarrow \mathcal{A} \xrightarrow{d} \Omega^1\xrightarrow{d} \cdots \Omega^k\xrightarrow{d}\cdots
\end{equation*}
with respect to a coboundary operator $d$ which obeys the {\em graded Leibniz rule} $d(a\wedge b)=d(a)\wedge b + (-1)^{|{\rm deg}(a)|} a\wedge d(b)$, for all pair of homogeneous elements $a,b\in\Omega$. In particular, $d:\mathcal{A}\to \Omega^{1}$ is a $\Omega^{1}$-valued derivation of a $\mathbb{K}$-algebra $\mathcal{A}$.

The following definition gives the general features of differential calculi that we study in this paper.

\begin{definition}(\cite{BS14}, p. 416)\label{dimensioncalculus}
We say that a differential calculus $(\Omega\mathcal{A},d)$, where $\Omega\mathcal{A}=\bigoplus_{i\in\mathbb{N}}\Omega^{i}\mathcal{A}$, with $\Omega^{i}\mathcal{A}=\bigwedge_{j=1}^{i}\Omega^{1}\mathcal{A}$ and $\Omega^{0}\mathcal{A}=\mathcal{A}$, is of \textit{dimension} $n$, if $\Omega^{n}\mathcal{A}\not=0$, and $\Omega^{m}=0$, for all $n<m$. We say that $\Omega^{n}\mathcal{A}$ \textit{admits a volume form} if $\Omega^{n}\mathcal{A}$ is isomorphic to $\mathcal{A}$ as a right and left module (not necessary as a bimodule). If this is the case, and $\omega$ is a right generator of $\Omega^{n}\mathcal{A}$, we say that $\omega\in \Omega^{n}\mathcal{A}$ is a \textit{volume form}.
We call a differential calculus $(\Omega,d)$ \textit{connected} if $\text{Ker}(d|_{\Omega^{0}})=\mathbb{K}$.  
\end{definition}

For a differential calculus that admits a volume form, we define for each right generator $\omega\in\Omega^{n}\mathcal{A}$, the {\em algebra automorphism} $\nu_{\omega}:\mathcal{A}\rightarrow\mathcal{A}$ such that for each $a\in\mathcal{A}$, $a\omega=\omega\nu_{\omega}(a)$, and the $\mathcal{A}$-module isomorphism $\pi_{\omega}:\Omega^{n}(\mathcal{A})\to \mathcal{A}$, such that for all $a\in \mathcal{A}$ we have $\pi_{\omega}(\omega a)=a$.

The following definition of an \textit{integrable differential calculus} seeks to portray a version of Hodge star isomorphisms between the complex of differential forms of a differentiable manifold and a complex of dual modules of it (Brzezi{\'n}ski \cite{Bre15}, p. 112).

\begin{definition}(\cite{BS14}, Definition 2.1)
An $n$-dimensional differential calculus $\Omega A$ is said to be \textit{integrable} if $\Omega A$ admits a complex of integral forms $(\mathcal{I}A,\nabla)$ (\cite{BKL10}, Section 2), where $\mathcal{I}_{i}\mathcal{A}=\text{Hom}_{\mathcal{A}}(\Omega^{i}(\mathcal{A}),\mathcal{A})$, for $0\leq i\leq n$, and for which there exist an algebra automorphism $\nu$ of $\mathcal{A}$ and $\mathcal{A}$-bimodule isomorphism $\Theta_{k}:\Omega^{k}\mathcal{A}\rightarrow \mathcal{I}_{n-k}\mathcal{A}$, $k=0,\dotsc ,n$, such that for each $k$, we have the commutative diagram given by 
\begin{center}
\begin{tikzcd}
\Omega^{k}\mathcal{A} \arrow[r, "d"]\arrow[d, "\Theta_{k}"]& \Omega^{k+1}\mathcal{A}\arrow[d, "\Theta_{k+1}"]\\
\mathcal{I}_{n-k}\mathcal{A}\arrow[r, "\nabla_{k}"] & \mathcal{I}_{n-(k+1)}\mathcal{A}
\end{tikzcd}
\end{center}
\end{definition}

The following result allows us to guarantee the integrability of a differential calculus $\Omega(\mathcal{A})$ by considering the existence of finitely generator elements that allow to determine left and right components of any homogeneous element of $\Omega(\mathcal{A})$, using the volume isomorphisms $\nu_{\omega}$ and $\pi_{\omega}$. 

\begin{lemma}{\rm (\cite{BS14}, Lemma 2.7)}\label{chayanne} Let $\Omega(\mathcal{A})$ be an $n$-dimensional calculus over $\mathcal{A}$ admitting a volume form $\omega$. Assume that for all $k=1,\dotsc,n-1$, there exists a finite number of forms $\omega_{i}^{k},\overline{\omega}_{i}^{k}\in \Omega^{k}(\mathcal{A})$ such that for all $\omega'\in \Omega^{k}(\mathcal{A})$,
\begin{align}
    \omega'=\sum_{i}\omega^{k}_{i}\pi_{\omega}(\overline{\omega}_{i}^{n-k}\wedge \omega')=\sum_{i}\nu_{\omega}^{-1}( \pi_{\omega}(\omega'\wedge \omega_{i}^{n-k}))\overline{\omega}_{i}^{k}. \label{integrablelem}
\end{align}
Then $\omega$ is an integral form and all the $\Omega^{k}(\mathcal{A})$ are finitely generated and projective as left and right $\mathcal{A}$-modules.
\end{lemma}

Now, we recall the key notion for the paper.
\begin{definition}(\cite{BS14}, p. 421)\label{deffsmooth}
An affine algebra $R$ of integer Gelfand-Kirillov dimension $m$ is said to be \textit{differentially smooth} if there exists a connected, $m$-dimensional, integrable differential calculus on $R$.
\end{definition}

About the differential smoothness, from Brzezi{\'n}ski \cite{Bre15} and \cite{BL18} we know that this property holds for some tensor products and skew polynomial rings (introduced by Ore \cite{Ore1933}) of automorphism type $R[y;\sigma,\delta]$ over the polynomial ring $R=\mathbb{C}[x]$ with a non-trivial derivation $\delta$ over $R$. Precisely, related to these objects, we recall below the first family of noncommutative rings that interest us in this work.
\begin{definition}[\cite{BellSmith1990}; \cite{Rosenberg1995}, Definition C4.3]\label{monito}
A $3$-{\em dimensional algebra} $\mathcal{A}$ is a $\mathbb{K}$-algebra generated by the indeterminates $x, y, z$ subject to the relations $yz-\alpha zy=\lambda$, $zx-\beta xz=\mu$, and $xy-\gamma yx=\nu$, where $\lambda,\mu,\nu \in \mathbb{K}x+\mathbb{K}y+\mathbb{K}z+\mathbb{K}$, and $\alpha, \beta, \gamma \in \mathbb{K}^{*}$. $\mathcal{A}$ is called a \textit{3-dimensional skew polynomial $\mathbb{K}$-algebra} if the set $\{x^iy^jz^k\mid i,j,k\geq 0\}$ forms a $\mathbb{K}$-basis of the algebra.
\end{definition}

Different authors have studied ring-theoretical and computational properties of 3-dimensional skew polynomial algebras (see \cite{LFGRSV}, \cite{GoodearlLetzter1994}, \cite{Jordan1993}, \cite{Jordan1995}, \cite{JordanWells1996},  \cite{LarssonSilvestrov2007}, \cite{Smith1991},  \cite{RedmanPhD1996}, \cite{Redman1999}, and \cite{ReyesSuarez2017a}). Next, we recall the classification of these algebras.

\begin{proposition}[\cite{Rosenberg1995}, Theorem C4.3.1]\label{quasipolynomial}
 Up to isomorphism, a 3-dimensional skew polynomial $\mathbb{K}$-algebra $\mathcal{A}$ is given by the following relations:
\begin{enumerate}
    \item [\rm (1)] If $|\{\alpha,\beta,\gamma\}|=3$, then $\mathcal{A}$ is defined by the relations $yz-\alpha zy=0$, $zx-\beta xz=0$, and $xy-\gamma yx=0$.
    \item [\rm (2)] If $|\{\alpha,\beta,\gamma\}|=2$ and $\beta\not=\alpha=\gamma=1$, then $\mathcal{A}$ is defined by one of the following rules:
    \begin{enumerate}
        \item [\rm (a)] $yz-zy=z$, $zx-\beta xz=y$, and $xy-yx=x$.
        \item [\rm (b)] $yz-zy=z$, $zx-\beta xz=b$, and $xy-yx=x$.
        \item [\rm (c)] $yz-zy=0$, $zx-\beta xz=y$, and $xy-yx=0$.
        \item [\rm (d)] $yz-zy=0$, $zx-\beta xz=b$, and $xy-yx=0$.
        \item [\rm (e)] $yz-zy=az$, $zx-\beta xz=0$, and $xy-yx=x$.
        \item [\rm (f)] $yz-zy=z$, $zx-\beta xz=0$, and $xy-yx=0$.
    \end{enumerate}
    Here $a,b\in \mathbb{K}$ are arbitrary; all nonzero values of $b$ yield isomorphic algebras.
    \item [\rm (3)] If $\alpha=\beta=\gamma\not=1$, and if $\beta\not=\alpha=\gamma\not= 1$, then $\mathcal{A}$ is one of the following algebras:
    \begin{enumerate}
        \item [\rm (a)] $yz-\alpha zy=0$, $zx-\beta xz=y+b$, and $xy-\alpha yx=0$.
        \item [\rm (b)] $yz-\alpha zy=0$, $zx-\beta xz=b$, and $xy-\alpha yx=0$.
    \end{enumerate}
    Here $a,b\in \mathbb{K}$ is arbitrary; all nonzero values of $b$ yield isomorphic algebras.
    \item [\rm (4)] If $\alpha=\beta=\gamma\not=1$, then $\mathcal{A}$ is determined by the relations $yz-\alpha zy= a_1 x + b_1$, $zx-\alpha xz= a_2 y + b_2$, and $xy-\alpha yx= a_3 z + b_3$.

    If $a_i=0$, then all nonzero values of $b_i$ yield isomorphic algebras.
    \item [\rm (5)] If $\alpha=\beta=\gamma=1$, then $\mathcal{A}$ is isomorphic to one of the following algebras:
    \begin{enumerate}
        \item [\rm (a)] $yz-zy=x$, $zx-xz=y$, and $xy-yx=z$.
        \item [\rm (b)] $yz-zy=0$, $zx-xz=0$, and $xy-yx=z$.
        \item [\rm (c)] $yz-zy=0$, $zx-xz=0$, and $xy-yx=b$.
        \item [\rm (d)] $yz-zy=-y$, $zx-xz=x+y$, and $xy-yx=0$.
        \item [\rm (e)] $yz-zy=az$, $zx-xz=x$, and $xy-yx=0$.
    \end{enumerate}
    Here $a,b\in\mathbb{K}$ are arbitrary; all nonzero values of $b$ yield isomorphic algebras.
\end{enumerate}
\end{proposition}

Now, we are going to formulate the first result of the article, Theorem 1. This allows us to study the differential smoothness of some 3-dimensional skew polynomial algebras.

\begin{theorem}\label{Theorem3d}
Let $\mathcal{A}$ be the $\mathbb{K}$-algebra generated by the indeterminates $x, y$ and $z$ subject to the relations given by
\begin{align*}
    yz-z(\alpha y+a)=0, \ \ \ zx-\beta xz = b,\ \ \  xy-(\gamma y+d)x=0, 
\end{align*}
where $\gamma, \alpha,\beta\in \mathbb{K}\setminus \{0\}$, $b\in \mathbb{K}$, such that {\rm (i)} $d=a=0$ or {\rm (ii)} $\alpha=\gamma=1$. If ${\rm GKdim}(\mathcal{A})=3$, then $\mathcal{A}$ is differentially smooth. 
\end{theorem}
\begin{proof}
In both cases (i) and (ii), one can check that the automorphisms of $\mathcal{A}$ given by
 \begin{align*}
           \nu_{x}(x)&=\beta^{-1}x,\ \ \ \ \nu_{x}(y)=\gamma^{-1}y-d,\ \ \ \   \nu_{x}(z)=\beta z,\\ 
           \nu_{y}(x)&=\gamma x,\ \ \ \ \ \ \ \nu_{y}(y)=y,\ \ \ \ \ \ \  \ \ \ \ \ \ \ \nu_{y}(z)=\alpha^{-1}z,\\ 
           \nu_{z}(x)&=\beta^{-1}x,\ \ \ \  \nu_{z}(y)=\alpha y+a,\ \ \ \ \ \ \  \nu_{z}(z)=\beta z, 
        \end{align*}
commute with each other. Having in mind the left structure defined with these automorphisms $\nu_w,\ w = x, y, z$, over the right free $\mathcal{A}$-module $\Omega^{1}(\mathcal{A}):=d(x)\mathcal{A}\bigoplus d(y)\mathcal{A}$ $\bigoplus d(z)\mathcal{A}$, with left action defined, for all $a\in \mathcal{A}$, by $ad(w)=d(w)\nu_{w}(a)$, and since the extension of the assignments $w\mapsto d(w)$ to a derivative $d:\mathcal{A}\rightarrow \Omega^{1}(\mathcal{A})$ is well-defined by the definition of $\nu_{w}$, if $\nu_{w}(w):=\theta_{w}w$, then for all $i\in\mathbb{N} \setminus \{0\}$, one can assert that $d(w^{i})=d(w)(\sum_{j=0}^{i-1}\theta_{w}^{j})w^{i-1}$. Note that if $\theta_{w}=1$, then $d(w^{i})=d(w)\partial_{w}(w^{i})$. In this way, we have the well-defined $\mathbb{K}$-maps $\overline{\partial}_{w}:\mathcal{A}\rightarrow \mathcal{A}$, such that $d(a)=\sum_{w\in \{x,y,z\}}d(w)\overline{\partial}_{w}(a)$, where for all $x^{i}y^{j}z^{k}\in \mathcal{A}$,

\begin{align*}
    d(x^{i}y^{j}z^{k})&=d(x)\biggl(\sum_{t=0}^{i-1}\theta_{x}^{t}\biggr)x^{i-1}y^{j}z^{k}+d(y)\biggl(\sum_{t=0}^{j-1}\theta_{y}^{t}\biggr)\nu_{y}(x^{i})y^{j-1}z^{k}\\&+d(z)\biggl(\sum_{t=0}^{k-1}\theta_{z}^{t}\biggr)\nu_{z}(x^{i})\nu_{z}(y^{j})z^{k-1}.
\end{align*}

Note that $a\in {\rm Ker}(d)$ if and only if $a\in {\rm Ker}(\overline{\partial}_{x})\cap {\rm Ker}(\overline{\partial}_{y})\cap {\rm Ker}(\overline{\partial}_{z})=\mathbb{K}$, which shows that $(\Omega(\mathcal{A}),d)$ is connected, where $\Omega(\mathcal{A})=\bigoplus_{i=0}^{3}\Omega^{i}(\mathcal{A})$. In $\Omega^{2}(\mathcal{A})$ we get that, in both cases (i) and (ii), the following relations hold:
\begin{align*}
    d(y)\wedge d(x)&=-\gamma^{-1}d(x)\wedge d(y),\\ d(z)\wedge d(x)&=-\beta d(x)\wedge d(z),\\
    d(z)\wedge d(y)&=-\alpha^{-1}d(y)\wedge d(z),
\end{align*}
Since the automorphisms $\nu_w, w\in \{x, y, z\}$,  commute with each other, there are no additional relationships to the previous ones, so that $\Omega^{2}(\mathcal{A})=d(x)\wedge d(y)\mathcal{A}\bigoplus d(x)\wedge d(z)\mathcal{A}\bigoplus d(y)\wedge d(z)\mathcal{A}$. Thus, $\Omega^{3}=\omega \mathcal{A}\cong \mathcal{A}$ as a right and left $\mathcal{A}$-module, with $\omega:=d(x)\wedge d(y) \wedge d(z)$, where $\nu_{\omega}=\nu_{x}\circ \nu_{y}\circ \nu_{z}$, i.e., $\omega$ is a volume form of $\mathcal{A}$. From Lemma \ref{chayanne}, we get that $\omega$ is an integral form by setting
\begin{align*}
            \omega_{1}^{1}&=\overline{\omega}_{1}^{1}=d(x),\ \ \ \ \ \ \ \ \ \ \  \omega_{2}^{1}=\overline{\omega}_{2}^{1}=d(y), \ \ \ \ \ \ \ \ \ \ \ 
            \omega_{3}^{1}=\overline{\omega}_{3}^{1}=d(z)\\
            \omega_{1}^{2}&=d(y)\wedge d(z), \ \ \ \ \ \ \ \ \ \  \omega_{2}^{2}=- \gamma d(x)\wedge d(z), \ \ \ \ \ \omega_{3}^{2}=\alpha\beta^{-1}d(x)\wedge d(y),\\
            \overline{\omega}_{1}^{2}&=\gamma\beta^{-1} d(y)\wedge d(z),\ \ \ \overline{\omega}_{2}^{2}=-\alpha d(x)\wedge d(z), \ \ \ \ \ \overline{\omega}_{3}^{2}=d(x)\wedge d(y).
        \end{align*}
        Of course, this shows that $\mathcal{A}$ is differentially smooth.
\end{proof}

\begin{remark}\label{noDS}
Suppose that for some generators $x_{i},x_{j}$ and $x_{k}$ of an algebra $\mathcal{A}$ generated by the indeterminates $x_{1},\dotsc ,x_{n}$ one has the relationships given by $x_{i}x_{j}-ax_{j}x_{i}=bx_{i}+cx_{j}+fx_{k}+e$. If we have a first order differential calculus $(\Omega^{1},d)$ and $d$ is a well defined derivation $d:\mathcal{A}\rightarrow\Omega^{1}$, we get that 
 \begin{align*}
     d(x_{i})x_{j}+x_{i}d(x_{j})-ad(x_{j})x_{i}-ax_{j}d(x_{i})=bd(x_{i})+cd(x_{j})+fd(x_{k})+e,
 \end{align*}
whence $d(x_{k})$ is generated in the $\mathcal{A}$-bimodule by the elements $d(x_{i})$ and $d(x_{j})$. Since $\Omega^{1}$ is generated as $\mathcal{A}$-bimodule by $d(\mathcal{A})$, then $\Omega^{1}$ is generated by $n-1$ elements, and hence $\Omega^{n}=\bigwedge_{i=1}^{n} \Omega^{1}=0$. Besides, if $\text{GKdim}(\mathcal{A})=n$, we get that $\mathcal{A}$ cannot be differentially smooth because there is no a differential graded calculus of dimension $n$.  
\end{remark}

From Theorem \ref{Theorem3d}, we can  guarantee the differential smoothness of 3-dimensional skew polynomial $\mathbb{K}$-algebras appearing in $(1), (2)(\text{b}), (2)(\text{d}), (2)(\text{e})$, $(2)(\text{f})$, $(3)(\text{b})$ and $(5)(\text{c})$, and by other hand, Remark \ref{noDS} shows us that the algebras $(2)(\text{a})$, $(2)(\text{c})$, $(3)(\text{a})$, $(4)$ with $|a_{1}|+ |a_{2}|+|a_{3}|>0$, $(5)(\text{a})$, $(5)(\text{b})$ and $(5)(\text{d})$, are no differentially smooth. For the algebra $(5)$(c), though this algebra does not hold the conditions of Theorem \ref{Theorem3d}, we have that the automorphisms $\nu_{w}:={\rm id}_{\mathcal{A}}$ guarantee a structure in the sense of Theorem \ref{Theorem3d}. However, for the algebra (5)(e), we cannot conclude its differential smoothness using this reasoning because applying $d$ to $zx-x(z+1)=0$ we obtain the following facts:
\begin{align*}
    d(z)x+zd(x)-d(x)(z+1)-xd(z)&=0\\
    d(z)x+d(x)\nu_{x}(z)-d(x)(z+1)-d(z)\nu_{z}(x)&=0\\
    d(z)[x-\nu_{z}(x)]+d(x)[\nu_{x}(z)-(z+1)]&=0.
\end{align*}
Due to the right free structure of $\Omega^{1}$,  necessarily $\nu_{z}(x)=x$ and $\nu_{x}(z)=z+1$. In the same way, since $xy-yx=0$ then $\nu_{x}(y)=y$ and $\nu_{y}(x)=x$. Thus, $\nu_{x}$ respect $zy-(y+a)z=0$ if and only if $a=0$. In other words, if $a\not=0$, we cannot define a left structure on $\Omega^{1}(\mathcal{A})$ such that there exists a well-defined derivation $d:\mathcal{A}\rightarrow\Omega^{1}(\mathcal{A})$, which is our main trouble about the differential smoothness of algebra (5)(e). 

\subsection{Special $n$-dimensional skew polynomial algebras}\label{sec:4}

In the proof of Theorem \ref{Theorem3d} we saw that we need automorphisms $\nu_{x_{i}}$ of $\mathcal{A}$, for each generator $x_{i}$, such that the following conditions hold:
 \begin{itemize}
     \item  the automorphisms allow to define a differential $d:\mathcal{A}\rightarrow\Omega^{1}(\mathcal{A})$ under the left action $ad(x_{i})=d(x_{i})\nu_{x_{i}}(a)$, for $a\in\mathcal{A}$, and $1\leq i\leq n$;  
     \item for any pair of generators $x_{i}$ and $x_{j}$, $[\nu_{x_{i}},\nu_{x_{j}}]=0$.
 \end{itemize}
 
In a general way, if $\mathcal{A}$ is an algebra generated by a set of indeterminates $\{x_{1},\dotsc, x_{n}\}$ satisfying the relations  
 \begin{align}
     x_{i}x_{j}-a_{ij}x_{j}x_{i}=b_{ij}x_{i}+c_{ij}x_{j}+e_{ij}, \ \ {\rm for\ all}\ 1\leq i<j\leq n, \label{spag}
 \end{align}
 
where $a_{ij}, b_{ij}, c_{ij}, e_{ij}\in \mathbb{K},\ a_{ij}\not=0$, if we want to define $d$ we need that $\nu_{x_{i}}(x_{j})=a_{ij}^{-1}x_{j}-a_{ij}^{-1}b_{ij}$ and $\nu_{x_{j}}(x_{i})=a_{ij}x_{i}+c_{ij}$, for $i<j$. Thus, for any $x_{k}$, if $j>k$ then we have that $\nu_{x_{j}}(x_{k})=a_{kj}x_{k}+c_{kj}$, and if $j<k$ then $\nu_{x_{j}}(x_{k})=a_{jk}^{-1}x_{k}-a_{jk}^{-1}b_{jk}$. In order to establish if the automorphisms $\nu$'s are well defined, we proceed in the following way: 
 
\begin{enumerate}
    \item  If $j<k$, we have $\nu_{x_{k}}(x_{j})=a_{jk}x_{j}+c_{jk}$, and since $x_{j}x_{k}-a_{jk}x_{k}x_{j}=b_{jk}x_{j}+c_{jk}x_{k}+e_{jk}$, when we apply $\nu_{x_{k}}$ we obtain $ (a_{jk}x_{j}+c_{jk})\nu_{x_{k}}(x_{k})-a_{jk}\nu_{x_{k}}(x_{k})(a_{jk}x_{j}+c_{jk}) = b_{jk}(a_{jk}x_{j}+c_{jk})+c_{jk}\nu_{x_{k}}(x_{k})+e_{jk}$.
    
 If $\nu_{x_{k}}(x_{k})=a_{kk}x_{k}-b_{kk}$, we get the expression 
$0x_{k}+(a_{jk}b_{kk}(a_{jk}-1)+a_{jk}b_{jk}(a_{kk}-1))x_{j} + (a_{kk}a_{jk}-1)e_{jk}+(a_{jk}b_{kk}-b_{jk})c_{jk}=0$.
In this way,
\begin{align}
    b_{kk}(a_{jk}-1)+b_{jk}(a_{kk}-1)&=0,\label{eq1}\\ (a_{kk}a_{jk}-1)e_{jk}+(a_{jk}b_{kk}-b_{jk})c_{jk}&=0.\nonumber
\end{align}
 \item If $k<j$, we have $\nu_{x_{k}}(x_{j})=a_{kj}^{-1}x_{j}-a_{kj}^{-1}b_{kj}$, and since $x_{k}x_{j}-a_{kj}x_{j}x_{k}=b_{kj}x_{k}+c_{kj}x_{j}+e_{kj}$, when we apply $\nu_{x_{k}}$ we get $\nu_{x_{k}}(x_{k})(a_{kj}^{-1}x_{j}-a_{kj}^{-1}b_{kj})-a_{kj}(a_{kj}^{-1}x_{j}-a_{kj}^{-1}b_{kj})\nu_{x_{k}}(x_{k}) = b_{kj}\nu_{x_{k}}(x_{k})+c_{kj}(a_{kj}^{-1}x_{j}-a_{kj}^{-1}b_{kj})+e_{kj}$. 
 
 If $\nu_{x_{k}}(x_{k})=a_{kk}x_{k}-b_{kk}$, then 
\begin{align*}
     &(a_{kk}a_{kj}^{-1}c_{kj}-b_{kk}a_{kj}^{-1}+b_{kk}-c_{kj}a_{kj}^{-1})x_{j}\\&+a_{kk}a_{kj}^{-1}e_{kj}+b_{kk}a_{kj}^{-1}b_{kj}+c_{kj}a_{kj}^{-1}b_{kj}-e_{kj}=0,
\end{align*}
 which is equivalent to the equations
 \begin{align}
     (a_{kk}-1)c_{kj}+b_{kk}(a_{kj}-1)&=0,\label{eq2}\\e_{kj}(a_{kk}-a_{kj})+(b_{kk}+c_{kj})b_{kj}&=0.\nonumber
 \end{align}
 \end{enumerate}
Now, for a relation that do not involve $x_{k}$, we mean a relation between $x_{j}$ and $x_{t}$ (without lost of generality with $j<t$), i.e., $x_{j}x_{t}-a_{jt}x_{t}x_{j}=b_{jt}x_{j}+c_{jt}x_{t}+e_{jt}$, if we apply $\nu_{x_{k}}$ we get the following three cases:
\begin{enumerate}
    \item If $j<t<k$, we have that  $\nu_{x_{k}}(x_{j})=a_{jk}x_{j}-c_{jk}$ and $\nu_{x_{k}}(x_{t})=a_{tk}x_{t}-c_{tk}$, whence  
    \begin{align*}
        &\ \ \ \ (a_{jt}c_{tk}a_{jk}-a_{jk}c_{tk}-b_{jt}a_{jk}+a_{jk}a_{tk}b_{jt})x_{j}\\&+(a_{jt}a_{tk}c_{jk}-c_{jk}a_{tk}-c_{jt}a_{tk}+a_{jk}a_{tk}c_{jt})x_{t}\\&+a_{jk}a_{tk}e_{jt}+c_{jk}c_{tk}-a_{jt}c_{tk}c_{jk}+b_{jt}c_{jk}+c_{jt}c_{tk}-e_{jt}=0,
    \end{align*}
    which is equivalent to the following equalities:
    \begin{align}
        (a_{jt}-1)c_{tk}+(a_{tk}-1)b_{jt}&=0,\nonumber\\  (a_{jt}-1)c_{jk}+(a_{jk}-1)c_{jt}&=0,\label{eq3}\\(a_{jk}a_{tk}-1)e_{jt}+b_{jt}c_{jk}+c_{jt}c_{tk}+(1-a_{jt})c_{tk}c_{jk}&=0. \nonumber
    \end{align}
    
    \item If $j<k<t$, we have that  $\nu_{x_{k}}(x_{j})=a_{jk}x_{j}-c_{jk}$ and $\nu_{x_{k}}(x_{t})=a_{kt}^{-1}x_{t}-a_{kt}^{-1}b_{kt}$, so
    \begin{align*}
    &\ \ \ \ (a_{jt}a_{kt}^{-1}c_{jk}-c_{jk}a_{kt}^{-1}-c_{jt}a_{kt}^{-1}+a_{jk}a_{kt}^{-1}c_{jt})x_{t}\\&+(a_{jt}a_{kt}^{-1}b_{kt}a_{jk}-b_{jt}a_{jk}-a_{jk}a_{kt}^{-1}b_{kt}+a_{jk}a_{kt}^{-1}b_{jt})x_{j}\\&-e_{jt}+b_{jt}c_{jk}+c_{jt}a_{kt}^{-1}b_{kt}+c_{jk}a_{kt}^{-1}b_{kt}-a_{jt}a_{kt}^{-1}b_{kt}c_{jk}+a_{jk}a_{kt}^{-1}e_{jt}=0,
    \end{align*}
that is, 
    \begin{align}
    (a_{jt}-1)c_{jk}+(a_{jk}-1)c_{jt}&=0,\nonumber\\(a_{jt}-1)b_{kt}+b_{jt}(1-a_{kt})&=0,\label{eq4}\\(a_{jk}-a_{kt})e_{jt}+(c_{jt}+c_{jk})b_{kt}+(b_{jt}a_{kt}-a_{jt}b_{kt})c_{jk}&=0.\nonumber
    \end{align}
    \item If $k<j<t$, we have that $\nu_{x_{k}}(x_{j})=a_{kj}^{-1}x_{j}-a_{kj}^{-1}b_{kj}$ and $\nu_{x_{k}}(x_{t})=a_{kt}^{-1}x_{t}-a_{kt}^{-1}b_{kt}$, which implies that 
    {\small\begin{align*}
    &\ \ \ \ (b_{jt}-b_{kt}+a_{jt}b_{kt}-b_{jt}a_{kt})x_{j}\\&+(a_{jt}b_{kj}-b_{kj}-c_{jt}a_{kj}+c_{jt})x_{t}\\&+e_{jt}+b_{kj}b_{kt}-a_{jt}b_{kt}b_{kj}+a_{kt}b_{jt}b_{kj}+a_{kj}c_{jt}b_{kt}-a_{kt}a_{kj}e_{jt}=0,
    \end{align*}}
    or what is the same,
    \begin{align}
    (a_{jt}-1)b_{kt}+(1-a_{kt})b_{jt}&=0,\nonumber\\b_{kj}(a_{jt}-1)+(1-a_{kj})c_{jt}&=0,\label{eq5}\\(1-a_{kj}a_{kt})e_{jt}+b_{kt}(b_{kj}+a_{kj}c_{jt})+b_{kj}(a_{kt}b_{jt}-a_{jt}b_{kt})&=0.\nonumber
    \end{align}
\end{enumerate}
 
 Additionally, if we want that the automorphisms $\nu$'s commute with each other, since for any $x_{k}$, if $j>k$, we have that $\nu_{x_{j}}(x_{k})=a_{kj}x_{k}-c_{kj}$, and if $j<k$, $\nu_{x_{j}}(x_{k})=a_{jk}^{-1}x_{k}-a_{jk}^{-1}b_{jk}$, and $\nu_{x_{k}}(x_{k})=a_{kk}x_{k}-b_{kk}$, we need the \textit{commutativity equations} given by the following:
 \begin{align}\label{commutativeequations}
      &\text{{\rm 1.} If } j,t>k, \text{ then } c_{kj}(a_{kt}-1)=c_{kt}(a_{kj}-1).\nonumber\\  
      &\text{{\rm 2.} If } j>k>t, \text{ then }  c_{kj}(1-a_{tk})=b_{tk}(a_{kj}-1).\nonumber\\
      &\text{{\rm 3.} If } j,t<k, \text{ then }  b_{jk}(1-a_{tk})=b_{tk}(1-a_{jk}).\\
      &\text{{\rm 4.} If } k<j, \text{ then }  c_{kj}(a_{kk}-1)=b_{kk}(a_{kj}-1).\nonumber\\
      &\text{{\rm 5.} If } j<k, \text{ then }  b_{kk}(1-a_{jk})=b_{jk}(a_{kk}-1).\nonumber
 \end{align}

 \begin{remark}
 Now, we have to guarantee that the differential calculus $(\Omega(\mathcal{A}),d)$, where $\Omega^{1}(\mathcal{A})=\bigoplus_{j=1}^{n}d(x_{j})\mathcal{A}$
 and $\Omega^{i}(A)=\bigwedge_{j=1}^{i}\Omega^{1}(\mathcal{A})$, is a differential connected calculus. Since $\nu_{x_{k}}(x_{k})=a_{kk}x_{k}-b_{kk}$, the following equalities hold: 
  \begin{align*}
      d(x_{k}^{i})&=\sum_{j=1}^{i}x_{k}^{j-1}d(x_{k})x_{k}^{i-j}=\sum_{j=1}^{i}d(x_{k})\nu_{x_{k}}(x_{k}^{j-1})x_{k}^{i-j}\\
      &=d(x_{k})\biggl(\sum_{j=1}^{i}(a_{kk}x_{k}-b_{kk})^{j-1}x_{k}^{i-j}\biggr)\\
      &=d(x_{k})\biggl(\sum_{j=1}^{i}\sum_{t=0}^{j-1}{j-1\choose t}a_{kk}^{t}x_{k}^{i+t-j}(-b_{kk})^{j-1-t}\biggr).
  \end{align*}
  With these facts, we have that $d(a)=\sum_{i=1}^{n}d(x_{i})\overline{\partial}_{i}(a)$, where
  {\small\begin{align*}
   \overline{\partial}_{i}(x_{1}^{l_{1}}\cdots x_{n}^{l_{n}})=\biggl(\prod_{j=1}^{i-1} \nu_{x_{i}}(x_{j}^{l_{j}})\biggr)\biggl(\sum_{j=1}^{l_{i}}\sum_{t=0}^{j-1}{j-1\choose t}a_{kk}^{t}x_{k}^{l_{i}+t-j}(-b_{kk})^{j-1-t}\biggr)x_{i+1}^{l_{i+1}}\cdots x_{n}^{l_{n}}.   
  \end{align*}}
  
If $i=1$, then for $\sum_{r\in\Gamma_{r}} \alpha_{r} x_{1}^{l_{1r}}\cdots x_{n}^{l_{nr}}\in \mathcal{A}$,

  {\small\begin{align*}
      \overline{\partial}_{1}\biggl(\sum_r \alpha_{r} x_{1}^{l_{1r}}\cdots x_{n}^{l_{nr}}\biggr)=\sum_r \alpha_{r} \biggl(\sum_{j=1}^{l_{1r}}\sum_{t=0}^{j-1}{j-1\choose t}a_{11}^{t}x_{1}^{l_{1r}+t-j}(-b_{11})^{j-1-t}\biggr)x_{2}^{l_{2r}}\cdots x_{n}^{l_{nr}}=0,
  \end{align*}}
  
  which is equivalent to
  
  $$\sum_{r'\in\Gamma_{r}} \alpha_{r} \biggl(\sum_{j=1}^{l_{1r}}\sum_{t=0}^{j-1}{j-1\choose t}a_{11}^{t}x_{1}^{l_{1r}+t-j}(-b_{11})^{j-1-t}\biggr)=0,
 $$ 
  
  where $r'\in\Gamma_{r} $ if and only if $x_{2}^{l_{2r}}\cdots x_{n}^{l_{nr}}=x_{2}^{l_{2r'}}\cdots x_{n}^{l_{nr'}}$, i.e., we can rewrite this expression as $\overline{\partial}_{1}(\sum_r \alpha_{r} x_{1}^{l_{1r}}\cdots x_{n}^{l_{nr}})=0$, and hence $\overline{\partial}_{1}(p_{r}(x_{1})x_{2}^{l_{2r}}\cdots x_{n}^{l_{nr}})=0$, for all $r$. If $p_{r}(x_{1})=\sum_{m=0}^{l_{1r}}q_{m}x_{1}^{m}$, with $q_{m}\in\mathbb{K}$, then
  \begin{align*}
      0&=\overline{\partial}_{1}(p_{r}(x)x_{2}^{l_{2r}}\cdots x_{n}^{l_{nr}})\\
      &=\sum_{m=0}^{l_{1r}}q_{m}\biggl(\sum_{j=1}^{m}(a_{11}x_{1}-b_{11})^{j-1}x_{1}^{m-j}\biggr)x_{2}^{l_{2r}}\cdots x_{n}^{l_{nr}}\\
      &=\sum_{m=0}^{l_{1r}}q_{m}\biggl(\sum_{j=1}^{m}\sum_{t=0}^{j-1}{j-1\choose t}a_{11}^{t}(-b_{11})^{j-1-t}x_{1}^{m+t-j}\biggr)x_{2}^{l_{2r}}\cdots x_{n}^{l_{nr}}
  \end{align*}
  Since in the last expression the unique way to obtain the highest exponent of $x_{1}$ is considering $m=l_{1r}, j=1$ and $t=0$, we have that $q_{l_{1r}}{1-1\choose 0}a_{11}^{0}b_{11}^{1-1-0}=q_{l_{1r}}=0$, and so $p_{r}(x_{1})\in\mathbb{K}$, for all $r$. In other words, we have that $a\in{\rm Ker}(\overline{\partial}_{1})={\rm gen}_{\mathbb{K}}\{x_{2},\dotsc,x_{n}\}$. Following this reasoning, we can check that ${\rm Ker}(\overline{\partial}_{1})\cap\cdots\cap {\rm Ker}(\overline{\partial}_{k})={\rm gen}_{\mathbb{K}}\{x_{k+1},\dotsc, x_{n}\}$, for each $1\leq k \leq n$, and thus ${\rm Ker}(d)={\rm Ker}(\overline{\partial}_{1})\cap\cdots \cap {\rm Ker}(\overline{\partial}_{n})=\mathbb{K}$.   
 \end{remark}
 
\begin{remark}

In order to apply Lemma \ref{chayanne} with the aim of obtaining an integrable differential calculus with volume form $\omega=d(x_{1})\wedge \cdots\wedge d(x_{n})$ and automorphism $\nu_{\omega}=\nu_{x_{1}}\circ\cdots\circ \nu_{x_{n}}$, for each injective increasing maps $\varphi_{i}:\{1,\dotsc ,k\}\rightarrow \{1,\dotsc ,n\}$ and $\overline{\varphi_{i}}:\{1,\dotsc ,(n-k)\}\rightarrow \{1,\dotsc ,n\}\setminus\text{Im}(\varphi_{i})$, that determine an $n$-permutations $\varphi_{i}\_\overline{\varphi}_{i}$ such that $\varphi_{i}\_\overline{\varphi}_{i}(p)=\varphi_{i}(p)$ if $p\leq k$, and $\varphi_{i}\_\overline{\varphi}_{i}(p)=\overline{\varphi}_{i}(p-k)$ if $k<p\leq n$, and analogously $\overline{\varphi}_{i}\_\varphi_{i}$, we define the elements $\omega_{i}^{k},\overline{\omega}_{i}^{k}\in\Omega^{k}(\mathcal{A})$ as,
\begin{align*}
    \omega^{k}_{i}&=A_{ik}d(x_{\varphi_{i}(1)})\wedge \cdots \wedge d(x_{\varphi_{i}(k)}),\\ \overline{\omega}^{n-k}_{i}&=\overline{A}_{i(n-k)}d(x_{\overline{\varphi}_{i}(1)})\wedge\cdots \wedge d(x_{\overline{\varphi}_{i}(n-k)}),
\end{align*}
where $A_{ik},\overline{A}_{ik}\in\mathbb{K}$ are defined (to obtain $\overline{\omega}_{i}^{n-k}\wedge\omega_{i}^{k}=\omega$) as follows: if $i<j$, then $x_{j}d(x_{i})=d(x_{i})\nu_{x_{i}}(x_{j})=d(x_{i})(a_{ij}^{-1}x_{j}-a_{ij}^{-1}b_{ij})$, and therefore applying $d$, $d(x_{i})\wedge d(x_{j})=-a_{ij}d(x_{j})\wedge d(x_{i})$. In this way,

\begin{itemize}
    \item if $\varphi_{i}(1)=1$, then
    \begin{align*}
        A_{ik}=\prod_{s=1}^{n-k}\prod_{t=1}^{\overline{\varphi}_{i}(s)-1}(-1)a_{t\overline{\varphi}_{i}(s)},\ \ \ \
        \overline{A}_{i(n-k)}=\prod_{s=1}^{n-k-1}\prod_{t=s+1}^{n-k}(-1)a^{-1}_{\overline{\varphi}_{i}(s)\overline{\varphi}_{i}(t)}.
    \end{align*}
    \item If $\overline{\varphi}_{i}(1)=1$, then
    \begin{align*}
        A_{ik}=\prod_{s=1}^{k-1}\prod_{t=s+1}^{k}(-1)a^{-1}_{\varphi_{i}(s)\varphi_{i}(t)},\ \ \ \
        \overline{A}_{i(n-k)}=\prod_{s=1}^{k}\prod_{t=\varphi_{i}(s)+1}^{\overline{\varphi}_{i}(n-k)}(-1)a_{\varphi_{i}(s)t}.
    \end{align*}
\end{itemize}

Since $\varphi_{i}(p):=i_{p}$, for $1\leq p\leq k$, and the set $\{d(x_{i_{1}})\wedge \cdots \wedge d(x_{i_{k}})\mid \ 1\leq i_{1}<\cdots<i_{k}\leq n \}$ forms a right (and left) base of $\Omega^{k}(\mathcal{A})$, for an arbitrary $\omega'=\sum_{1\leq i_{1}<\dotsb <i_{k}\leq n}d(x_{i_{1}})\wedge \dotsb \wedge d(x_{i_{k}}) a_{i_{1}\dotsc  i_{k}}\in\Omega^{k}(\mathcal{A})$, where $a_{i_{1}\dotsc i_{k}}\in\mathcal{A}$, we have the equalities 
\begin{align*}
    \omega^{k}_{i}\pi_{\omega}(\overline{\omega}_{i}^{n-k}\wedge \omega')&=\omega^{k}_{i}\pi_{\omega}(\overline{\omega}_{i}^{n-k}\wedge [d(x_{i_{1}})\wedge \cdots \wedge d(x_{i_{k}}) a_{i_{1}\cdots i_{k}}])\\
    &=d(x_{i_{1}})\wedge \cdots \wedge d(x_{i_{k}})\pi_{\omega}(\overline{\omega}_{i}^{n-k}\wedge \omega^{k}_{i} a_{i_{1}\cdots i_{k}})\\
    &=d(x_{i_{1}})\wedge \cdots \wedge d(x_{i_{k}}) a_{i_{1}\cdots i_{k}},
\end{align*}

\begin{align*}
    \nu_{\omega}^{-1}( \pi_{\omega}(\omega'\wedge \omega_{i}^{n-k}))\overline{\omega}_{i}^{k}&=\nu_{\omega}^{-1}( \pi_{\omega}(d(x_{i_{1}})\wedge \cdots \wedge d(x_{i_{k}}) a_{i_{1}\cdots i_{k}}\wedge \omega_{i}^{n-k}))\overline{\omega}_{i}^{k}\\
    &=\nu_{\omega}^{-1}( \pi_{\omega}(\overline{\omega}_{i}^{k} \wedge \omega_{i}^{n-k} b_{i_{1}\cdots i_{k}} ))d(x_{i_{1}})\wedge \cdots \wedge d(x_{i_{k}})\\
    &=\nu_{x_{\varphi_{i}(1)}}^{-1}\circ \cdots \circ \nu_{x_{\varphi_{i}(k)}}^{-1}(a_{i_{1}\cdots i_{k}})d(x_{i_{1}})\wedge \cdots \wedge d(x_{i_{k}}),
\end{align*}
where $b_{i_{1}\cdots i_{k}}:=\nu_{x_{\overline{\varphi}_{i}(1)}}\circ \cdots \circ \nu_{x_{\overline{\varphi}_{i}(n-k)}}(a_{i_{1}\cdots i_{k}})$. With this, we get that equations (\ref{integrablelem}) hold, and by Lemma \ref{chayanne}, $\Omega(\mathcal{A})$ is an integrable differential calculus.
\end{remark}

From \cite{LFGRSV}, Theorem 7.4.1, the algebra $\mathcal{A}$ defined by relations (\ref{spag}) has Gelfand-Kirillov dimension $n$, that is, ${\rm GKdim}(\mathcal{A}) = n$, so the following result follows.\\

\begin{theorem}\label{mostimportant}
Let $\mathcal{A}$ be the $\mathbb{K}$-algebra generated by $x_{1},\dotsc ,x_{n}$ subject to the relations given by 
 \begin{align*}
     x_{i}x_{j}-a_{ij}x_{j}x_{i}=b_{ij}x_{i}+c_{ij}x_{j}+e_{ij}, \ \ \ \text{where }a_{ij},b_{ij},c_{ij},e_{ij}\in \mathbb{K},\ a_{ij}\not=0, 
 \end{align*}    
 for all $1\leq i<j\leq n$. If relations given in  {\rm(\ref{eq1}), (\ref{eq2}), (\ref{eq3}), (\ref{eq4}), (\ref{eq5})} and commutativity equations (\ref{commutativeequations}) hold,  then $\mathcal{A}$ is differentially smooth.
\end{theorem}

From Theorem \ref{mostimportant} we get the differential smoothness of algebras such as the algebra of linear partial differential operators, the algebra of linear partial $q$-dilation operators, the additive analogue of Weyl algebra and the multiplicative analogue of the Weyl algebra (see \cite{LFGRSV}, Chapter 2, for a detailed definition of every algebra). 

From the above results, we obtain the following corollary.
 
\begin{corollary}\label{general}
Consider the Ore extension $\mathcal{A}=\mathbb{K}[x_{1},\dotsc ,x_{n}][y;\sigma, \delta]$ with $\sigma(x_{i})=b_{i}x_{i}+a_{i}$ and $\delta(x_{i})=c_{i}x_{i}$, satisfying one of the following conditions:
\begin{itemize}
    \item $a_{i}\not=0$ and $c_{i}=0$, for all $i=1,\dotsc ,n$; or 
    \item  $a_{i}=0$, for all $i=1,\dotsc ,n$, and $c_{i}(b_{k}-1)+c_{k}(b_{i}-1)=0$, for all pair of elements $i\not=k$. 
\end{itemize}

\noindent If ${\rm GKdim}(\mathcal{A})=n$, then $\mathcal{A}$ is a differentially smooth algebra.
\end{corollary}

\begin{remark}
  There is a natural question: Why do we choose for automorphism $\nu_{x_{k}}$ the assignment $\nu_{x_{k}}(x_{k})\in \mathbb{K}x_{k}+\mathbb{K}$? We make this choice because we want that the automorphisms $\nu$'s commute between them; of course, there are other possibilities. For instance, in \cite{MS88} we find that the automorphisms of $\mathbb{K}[x,y]$ are defined considering the assignments $x\mapsto x$, $y\mapsto y+h(x)$, or $x\mapsto a_{11}x + a_{12}y+a_{13}$, $y\mapsto a_{21}x+a_{22}y+a_{23}$, with $a_{11}a_{22}-a_{12}a_{21}\not=0$. If $\nu_{y}(x)=a_{11}x+a_{13}$, i.e., $a_{12}=0$, and $\nu_{y}(y)\not\in \mathbb{K}y+\mathbb{K}$, we have the following possibilities:
  
  \begin{enumerate}
      \item $\nu_{y}(x)=x$ and $\nu_{y}(y)=y+h(x)$ with $0\not=h(x)\in\mathbb{K}[x]$. In this case, $\nu_{x}(x)=a_{11}'x+a_{12}'y+a_{13}'$ and $\nu_{x}(y)=a_{21}'y+a_{23}'$. If $[\nu_{x},\nu_{y}]=0$, then $a_{21}'(y+h(x))+a_{23}'=a_{21}'y+a_{23}'+h(a_{11}'x+a_{12}'y+a_{13}')$, which implies that $$a_{21}'h(x)=h(a_{11}'x+a_{12}'y+a_{13}'),$$ whence $a_{21}'=1$, and $h(x)\in\mathbb{}$ or $\nu_{x}(x)=x$. 
      \item $\nu_{y}(x)=a_{11}x+a_{13}$ and $\nu_{y}(y)=a_{21}x+a_{22}y+a_{23}$ with $a_{21}\not=0$. In this way, for $\nu_{x}(x)=a_{11}'x+a_{12}'y+a_{13}'$ and $\nu_{x}(y)=a_{21}'y+a_{23}'$, if $[\nu_{x},\nu_{y}]=0$, we get that 
      \[
      a_{12}'a_{21}x+a_{12}'a_{22}y+a_{12}'a_{23}+a_{11}'a_{13}+a_{13}' = a_{11}a_{12}'y+a_{11}a_{13}'+a_{13}.
      \]
      
      This equality implies that $a_{12}'a_{21}=0$, $a_{12}'a_{22}=a_{11}a_{12}'$ and $a_{12}'a_{23}+a_{11}'a_{13}+a_{13}'=a_{11}a_{13}'+a_{13}$. Since $a_{21}\not=0$, we obtain $a_{12}'=0$ and $(a_{11}'-1)a_{13}=(a_{11}-1)a_{13}'$. With respect to $y$,
      \begin{align*}
          a_{21}(a_{11}'x+a_{13}')+a_{22}(a_{21}'y+a_{23}')+a_{23} = &\ a_{21}'(a_{21}x+a_{22}y+a_{23})+a_{23}'\\ 
         a_{21}a_{11}'x+a_{21}a_{13}'+a_{22}a_{23}'+a_{23} = &\ a_{21}'a_{21}x+a_{21}'a_{23}+a_{23}',
      \end{align*}
    whence $a_{21}a_{11}'=a_{21}'a_{21}$ and $a_{21}a_{13}'+(a_{22}-1)a_{23}'=(a_{21}'-1)a_{23}$. Since $a_{21}\not=0$, we obtain $a_{11}'=a_{21}'$.
  \end{enumerate}

In the case of automorphisms that commute with each other, it may be that $\nu_{y}(y)\not\in \mathbb{K}y+\mathbb{K}$. However, if we use polynomial expressions of degree greater than one, such as in Case 1, then the commutativity between the automorphisms may fail.
 \end{remark}

\section{Differential smoothness of diffusion algebras}\label{sec:5}

Diffusion algebras were introduced formally by Isaev et al., \cite{IPR01} as quadratic algebras that appear as algebras of operators that model the stochastic flow of motion of particles in a one dimensional discrete lattice. However, its origin can be found in Krebs and Sandow \cite{KS97}. 

Let us start by recalling its definition.

\begin{definition}{\rm (\cite{IPR01}, p. 5817)}\label{difftipo1}
The \textit{diffusion algebras type 1} are affine algebras $\mathcal{D}$ that are generated by $n$ intedeterminates $D_1,\dotsc, D_n$ over a field $\mathbb{K}$ that admit a linear PBW basis of ordered monomials of the form $D_{\alpha_{1}}^{k_{1}}D_{\alpha_{2}}^{k_{2}}\cdots D_{\alpha_{n}}^{k_{n}}$ with $k_{j}\in\mathbb{N}$ and $\alpha_1>\alpha_2>\cdots > \alpha_n$, and there exist elements $x_1\dotsc ,x_n\in \mathbb{K}$ such that for all $1\leq i<j\leq n$, there exist $\lambda_{ij}\in\mathbb{K}^*$ such that
\begin{equation}\label{equationsrule}
    \lambda_{ij}D_iD_j-\lambda_{ji}D_jD_i=x_jD_i-x_iD_j.
\end{equation}
\end{definition}

Fajardo et al., \cite{LFGRSV} studied ring-theoretical properties of a graded version of these algebras; its definition is the following: 

\begin{definition}{\rm (\cite{LFGRSV}, Section 2.4)}\label{difftipo2}
The \textit{diffusion algebras type 2} are affine algebras $\mathcal{D}$ generated by $2n$ variables $\{D_1,\dotsc ,D_n,x_1,\dotsc ,x_n\}$ over a field $\mathbb{K}$ that admit a linear PBW basis of ordered monomials of the form $B_{\alpha_{1}}^{k_{1}}B_{\alpha_{2}}^{k_{2}}\cdots B_{\alpha_{n}}^{k_{n}}$ with $B_{\alpha_{i}}\in \{D_1,\dotsc ,D_n,x_1,\dotsc ,x_n\}$, for all $i\leq 2n$, $k_{j}\in\mathbb{N}$ and $\alpha_1>\alpha_2>\cdots > \alpha_n$, such that, for all $1\leq i<j\leq n$, there exist elements $\lambda_{ij}\in \mathbb{K}^*$ such that
\begin{equation}\label{relations}
    \lambda_{ij}D_iD_j-\lambda_{ji}D_jD_i=x_jD_i-x_iD_j.
\end{equation}
\end{definition}

Different physical applications of algebras type 1 and 2 have been studied in the literature. From the point of view of ring-theoretical, homological and computational properties, several papers have been published (see \cite{LFGRSV}, \cite{Hamidizadehetal2020}, \cite{Hinchcliffe2005}, \cite{Levandovskyy2005}, \cite{LezamaVenegas2020}, \cite{NinoReyes2020}, \cite{ReyesRodriguez2021}, \cite{ReyesSuarez2019}, \cite{ReyesSuarez2021}, and \cite{Twarok2003}).

\begin{remark}
About the above definitions of diffusion algebras, we make the following comments:
\begin{itemize}
\item Isaev et al., \cite{IPR01} and Pyatov and Twarok \cite{PT02} defined diffusion algebras type 1 by taking $\mathbb{K} = \mathbb{C}$. Nevertheless, for the results obtained in this paper we can take any field not necessarily $\mathbb{C}$.
\item Hinchcliffe \cite{Hinchcliffe2005}, Definition 2.1.1, considered the following definition of diffusion algebras. Let $R$ be the algebra generated by $n$ indeterminates $x_1, x_2$, $\dotsc$, $x_n$ over $\mathbb{C}$ subject to relations $a_{ij}x_ix_j - b_{ij}x_jx_i = r_jx_i - r_ix_j$, whenever $i < j$, for some parameters $a_{ij}\in \mathbb{C}\ \backslash\  \{0\}$, for all $i < j$ and $b_{ij}, r_i \in \mathbb{C}$, for all $i < j$. He defined the {\em standard monomials} to be those of the form $x_n^{i_n}x_{n-1}^{i_{n-1}}\dotsb x_2^{i_2}x_1^{i_1}$. $R$ is called a {\em diffusion algebra} if it admits a {\em PBW} {\em basis of these standard monomials}. In other words, $R$ is a diffusion algebra if these standard monomials are a $\mathbb{C}$-vector space basis for $R$. In what follows, we will consider the notation presented in Definitions \ref{difftipo1} and \ref{difftipo2}, i.e., over any field $\mathbb{K}$. 
\item Following Krebs and Sandow \cite{KS97}, the relations (\ref{equationsrule}) are consequence of subtracting (quadratic) operator relations of the type
  \begin{equation*}\label{Prediff}
    \Gamma_{\gamma\delta}^{\alpha\beta}D_{\alpha}D_{\beta}=D_{\gamma}X_{\delta}-X_{\gamma}D_{\delta},\ \ \ \ \ \ \text{for all}\ \gamma,\delta=0,1,\dotsc ,n-1,
\end{equation*}
 where $\Gamma_{\gamma\delta}^{\alpha\beta}\in \mathbb{K}$, and $D_{i}$'s and $X_{j}$'s are operators of a particular vector space, such that not necessarily $[D_{i},X_{j}]=0$ holds (see \cite{KS97}, p. 3168, for more details). 
 \end{itemize}
\end{remark}

Next, we present some combinatorial properties appearing in diffusion algebras (Lemma \ref{combin1} and Proposition \ref{conmutaciondiffder}). It is very possible that these are found in the literature; however, we could not find them explicitly somewhere. Before, we consider the following definition.

\begin{definition}\label{pyq}
If $D_{i}$ and $D_{j}$ are generators of a diffusion algebra $\mathcal{D}$ (as in Definitions \ref{difftipo1} or \ref{difftipo2}), for all pair of elements $k,n\in \mathbb{N}$ such that $k\le n$, the numbers $P_{k}^{n}, Q_{k}^{n}\in \mathbb{K}$ are defined as
\begin{align*}
    P_{k}^{n}&=\sum_{t=1}^{k} {n-k+t-1\choose n-k} \lambda_{ji}^{t-1}\lambda_{ij}^{k-t},\ \ \ \ \ 
    Q_{k}^{n}={n\choose k-1} \lambda_{ji}^{k-1}, 
\end{align*}
where ${{n-k+t-1\choose n-k}}\in \mathbb{N}$ are the entries of the $k$-diagonal of the triangle formed with the first $n+1$ levels of the Pascal's triangle.
\end{definition}

\begin{lemma}\label{combin1}
If we consider elements $k, n\in \mathbb{N}$ such that $k \leq n$, then:
\begin{enumerate}
    \item[\rm 1.] \label{lemaa} $P_{k}^{n+1}=P_{k-1}^{n}\lambda_{ij}+Q_{k}^{n}$ and  $P_{n+1}^{n+1}=P_{n}^{n}\lambda_{ij}+\lambda_{ji}^{n}$.
    \item[\rm 2.] \label{lemab} $Q_{k}^{n+1}=Q_{k-1}^{n}\lambda_{ji}+Q_{k}^{n}$ and $Q_{n+1}^{n+1}=Q_{n}^{n}\lambda_{ji}+\lambda_{ji}^{n}$.
    \item[\rm 3.] \label{lemac} $P_{1}^{n}=Q_{1}^{n}=1$. 
\end{enumerate}
\end{lemma}

\begin{proof}
\begin{enumerate}
    \item As is well known, ${a+b\choose a}= {a+b\choose b}$, for all $a,b\in\mathbb{N}$, which implies the equalities
    {\begin{align*}
        P_{k-1}^{n}\lambda_{ij}+Q_{k}^{n}&=\sum_{t=1}^{k-1} {{n-k+1+t-1\choose n-k+1}} \lambda_{ji}^{t-1}\lambda_{ij}^{k-1-t+1}+{n\choose k-1} \lambda_{ji}^{k-1}\\
        &=\sum_{t=1}^{k-1} {{n+1-k+t-1\choose n+1-k}} \lambda_{ji}^{t-1}\lambda_{ij}^{k-t}\\
        &\quad +{n+1-k+k-1\choose k-1} \lambda_{ji}^{k-1}\\
        &=\sum_{t=1}^{k} {{n+1-k+t-1\choose n+1-k}} \lambda_{ji}^{t-1}\lambda_{ij}^{k-t}=P_{k}^{n+1},
    \end{align*}}
    and 
    \begin{align*}
        P_{n}^{n}\lambda_{ij}+\lambda_{ji}^{n}=\sum_{t=1}^{n}\lambda_{ji}^{t-1}\lambda_{ij}^{n-t+1}+\lambda_{ji}^{n}=P_{n+1}^{n+1}.
    \end{align*}
    \item Since ${n\choose k-2}+{n\choose k-1}={n+1 \choose k-1}$, for all $n,k\in\mathbb{N}$, we can assert that 
    \begin{align*}
        Q_{k-1}^{n}\lambda_{ji}+Q_{k}^{n}&= {n \choose k-2}\lambda_{ji}^{k-2}\lambda_{ji}+{n\choose k-1}\lambda_{ji}^{k-1}\\&=\biggl({n\choose k-2}+{n\choose k-1}\biggr)\lambda_{ji}^{k-1}={n+1 \choose k-1}\lambda_{ji}^{k-1}=Q_{k}^{n+1},
    \end{align*}
    and also,
    \begin{align*}
        Q_{n}^{n}\lambda_{ji}+\lambda_{ji}^{n}=\biggl({n\choose n-1}+{n \choose n}\biggr)\lambda_{ji}^{n}={n+1 \choose n}\lambda_{ji}^{n}=Q_{n+1}^{n+1}.
    \end{align*}
    \item It follows from a short computation.
\end{enumerate}
\end{proof}

\begin{proposition}\label{conmutaciondiffder}
If $D_i$ and $D_{j}$ are generators of the diffusion algebra $\mathcal{D}$ as in Definition \ref{difftipo1}, then for $n\geq 1$, 
{\small\begin{align}\label{conmutadora}
     \lambda_{ij}^{n}D_{i}^{n}D_{j}&= \lambda_{ji}^{n}D_{j}D_{i}^{n}+\sum_{k=1}^{n}(-1)^{k+n}\ P_{k}^{n}\ x_{i}^{n-k}x_{j}\ D_{i}^{k}+(-1)^{n+k-1}\ Q_{k}^{n}\ x_{i}^{n-k+1}\ D_{j}D_{i}^{k-1}.
\end{align}}
\end{proposition}
\begin{proof}
For $n=1$, it is clear that (\ref{conmutadora}) coincides with (\ref{equationsrule}). We suppose that the assertion holds for a fixed $n\in \mathbb{N}$. Then,
{\footnotesize \begin{align*}
     \lambda_{ij}^{n+1}D_{i}^{n+1}D_{j}&=\lambda_{ij}D_{i}\biggl( \lambda_{ji}^{n}D_{j}D_{i}^{n}+\sum_{k=1}^{n}(-1)^{k+n}\ P_{k}^{n}\ x_{i}^{n-k}x_{j}\ D_{i}^{k}\\&+(-1)^{n+k-1}\ Q_{k}^{n}\ x_{i}^{n-k+1}\ D_{j}D_{i}^{k-1}\biggr)\\
     &= \lambda_{ji}^{n}(\lambda_{ji}D_jD_i+x_jD_i-x_iD_j)D_{i}^{n}+\sum_{k=1}^{n}\biggl[(-1)^{k+n}\ P_{k}^{n}\ x_{i}^{n-k}x_{j}\ \lambda_{ij}D_{i}^{k+1}\\&+(-1)^{n+k-1}\ Q_{k}^{n}\ x_{i}^{n-k+1}\ (\lambda_{ji}D_jD_i+x_jD_i-x_iD_j)D_{i}^{k-1}\biggr]\\
     &= \lambda_{ji}^{n+1}D_jD_i^{n+1}+\lambda_{ji}^{n}x_jD_i^{n+1}-\lambda_{ji}^{n}x_iD_jD_{i}^{n}\\&+\sum_{k=1}^{n}\biggl[(-1)^{k+n}\ P_{k}^{n}\ x_{i}^{n-k}x_{j}\ \lambda_{ij}D_{i}^{k+1}+(-1)^{n+k-1}\ Q_{k}^{n}\ x_{i}^{n-k+1}\ \lambda_{ji}D_jD_{i}^{k}\\&+(-1)^{n+k-1}\ Q_{k}^{n}\ x_{i}^{n-k+1}\ x_jD_{i}^{k}+(-1)^{n+k}\ Q_{k}^{n}\ x_{i}^{n-k+2}\ D_jD_{i}^{k-1}\biggr]\\
\end{align*}}
By the associative law of the sum, and separating the term $k=n$ of the first sum and the term $k=1$ of the second, we obtain the equalities
{\small \begin{align*}
     \lambda_{ij}^{n+1}D_{i}^{n+1}D_{j}&= \lambda_{ji}^{n+1}D_jD_i^{n+1}+\lambda_{ji}^{n}x_jD_i^{n+1}-\lambda_{ji}^{n}x_iD_jD_{i}^{n}\\
     &+  P_{n}^{n}\ x_{j}\ \lambda_{ij}D_{i}^{n+1}-\ Q_{n}^{n}\ x_{i}\ \lambda_{ji}D_jD_{i}^{n}\\&+\sum_{k=1}^{n-1}\biggl[(-1)^{k+n}\ P_{k}^{n}\ x_{i}^{n-k}x_{j}\ \lambda_{ij}D_{i}^{k+1}+(-1)^{n+k-1}\ Q_{k}^{n}\ x_{i}^{n-k+1}\ \lambda_{ji}D_jD_{i}^{k}\biggr]\\&+\sum_{k=2}^{n}\biggl[(-1)^{n+k-1}\ Q_{k}^{n}\ x_{i}^{n-k+1}\ x_jD_{i}^{k}+(-1)^{n+k}\ Q_{k}^{n}\ x_{i}^{n-k+2}\ D_jD_{i}^{k-1}\biggl]\\
     &+(-1)^{n}\ Q_{1}^{n}\ x_{i}^{n}\ x_jD_{i}+(-1)^{n+1}\ Q_{1}^{n}\ x_{i}^{n+1}\ D_j
     \end{align*}}
     
     or equivalently,
     
     {\small \begin{align*}
     &= \lambda_{ji}^{n+1}D_jD_i^{n+1}+(\lambda_{ji}^{n}+P_{n}^{n}\lambda_{ij})x_jD_i^{n+1}-(\lambda_{ji}^{n}+\lambda_{ji}Q_{n}^{n})x_iD_jD_{i}^{n}\\&+\sum_{k=2}^{n}\biggl[(-1)^{k-1+n}\ P_{k-1}^{n}\ x_{i}^{n-k+1}x_{j}\ \lambda_{ij}D_{i}^{k}\\&+(-1)^{n+k-2}\ Q_{k-1}^{n}\ x_{i}^{n-k+2}\ \lambda_{ji}D_jD_{i}^{k-1}\biggr]\\&+\sum_{k=2}^{n}\biggl[(-1)^{n+k-1}\ Q_{k}^{n}\ x_{i}^{n-k+1}\ x_jD_{i}^{k}+(-1)^{n+k}\ Q_{k}^{n}\ x_{i}^{n-k+2}\ D_jD_{i}^{k-1}\biggr]\\
     &+(-1)^{n}\ P_{1}^{n+1}\ x_{i}^{n}\ x_jD_{i}+(-1)^{n+1}\ Q_{1}^{n+1}\ x_{i}^{n+1}\ D_j\\
     &= \lambda_{ji}^{n+1}D_jD_i^{n+1}+P_{n+1}^{n+1}x_jD_i^{n+1}-Q_{n+1}^{n+1}x_iD_jD_{i}^{n}\\&+\sum_{k=2}^{n}\biggl[(-1)^{k-1+n}\ (P_{k-1}^{n}\lambda_{ij}+Q_{k}^{n})\ x_{i}^{n-k+1}x_{j}\ D_{i}^{k}\\&+(-1)^{n+k}\ (Q_{k-1}^{n}\lambda_{ji}+Q_{k}^{n})\ x_{i}^{n-k+2}\ D_jD_{i}^{k-1}\biggr]\\
     &+(-1)^{n}\ P_{1}^{n+1}\ x_{i}^{n}\ x_jD_{i}+(-1)^{n+1}\ Q_{1}^{n+1}\ x_{i}^{n+1}\ D_j\\
     &= \lambda_{ji}^{n+1}D_jD_i^{n+1}+P_{n+1}^{n+1}x_jD_i^{n+1}-Q_{n+1}^{n+1}x_iD_jD_{i}^{n}\\&+\sum_{k=2}^{n}\biggl[(-1)^{k-1+n}\ P_{k}^{n+1}\ x_{i}^{n-k+1}x_{j}\ D_{i}^{k}+(-1)^{n+k}\ Q_{k}^{n+1}\ x_{i}^{n-k+2}\ D_jD_{i}^{k-1}\biggr]\\
     &+(-1)^{n}\ P_{1}^{n+1}\ x_{i}^{n}\ x_jD_{i}+(-1)^{n+1}\ Q_{1}^{n+1}\ x_{i}^{n+1}\ D_j,
\end{align*}}
where the second equality is due to a substitution on the index of the first sum and Lemma \ref{lemac} in the last term, and both last equations are due to the distributivity and Lemma \ref{lemaa}, parts 1 and 2. This proves that the assumption is true for $n+1$, which concludes the proof.
\end{proof}

\begin{remark}
 Since in diffusion algebras type 2 the generators $x_i$'s are central elements, for these algebras Proposition \ref{conmutaciondiffder} holds.
\end{remark}

Using a similar reasoning to the proof presented in Proposition \ref{conmutaciondiffder}, we can prove the following result.

\begin{proposition}\label{conmutaciondiffizq}
If $D_i$ and $D_{j}$ are generators of the diffusion algebra $\mathcal{D}$, then for all $n\geq 1$,  
\begin{align*}
     \lambda_{ij}^{n}D_{i}D_{j}^{n}&= \lambda_{ji}^{n}D_{j}^{n}D_{i}+\sum_{k=1}^{n}  Q_{k}^{n}\ x_{j}^{n-k+1}\ D_{j}D_{i}^{k-1}-P_{k}^{n}\ x_{j}^{n-k}x_{i}\ D_{i}^{k}.
\end{align*}
\end{proposition}

\begin{remark} 
Propositions \ref{conmutaciondiffder} and \ref{conmutaciondiffizq} allow to show the following facts. For the inner derivations, $\partial_{i},\partial_{j}:\mathcal{D}\rightarrow\mathcal{D}$ defined by $\partial_{k}(a)=D_{k}a-aD_{k}$ with $k=i,j$ and $i<j$, we have that
 {\scriptsize \begin{align*}
    \lambda_{ij}^{n}\partial_{j}(D_{i}^{n})&=\lambda_{ij}^{n}D_{i}^{n}D_{j}-\lambda_{ij}^{n}D_{j}D_{i}^{n}\\
    &=(\lambda_{ji}^{n}-\lambda_{ij}^{n})D_{j}^{n}D_{i} + \sum_{k=1}^{n}(-1)^{k+n}\ P_{k}^{n}\ x_{i}^{n-k}x_{j}\ D_{i}^{k}+(-1)^{n+k-1}\ Q_{k}^{n}\ x_{i}^{n-k+1}\ D_{j}D_{i}^{k-1}.
\end{align*}}

which means that for the basic element $D_{j}^{m}D_{i}^{n}$, 
{\footnotesize\begin{align*}
    \lambda_{ij}^{n}\partial_{j}(D_{j}^{m}D_{i}^{n})&=D_{j}^{m}\lambda_{ij}^{n}\partial_{j}(D_{i}^{n})\\
    &=D_{j}^{m}[(\lambda_{ji}^{n}-\lambda_{ij}^{n})D_{j}^{n}D_{i}\\&+\sum_{k=1}^{n}(-1)^{k+n}\ P_{k}^{n}\ x_{i}^{n-k}x_{j}\ D_{i}^{k}+(-1)^{n+k-1}\ Q_{k}^{n}\ x_{i}^{n-k+1}\ D_{j}D_{i}^{k-1}]\\
    &=(\lambda_{ji}^{n}-\lambda_{ij}^{n})D_{j}^{m+n}D_{i}\\&+\sum_{k=1}^{n}(-1)^{k+n}\ P_{k}^{n}\ x_{i}^{n-k}x_{j}\ D_{j}^{m}D_{i}^{k}+(-1)^{n+k-1}\ Q_{k}^{n}\ x_{i}^{n-k+1}\ D_{j}^{m+1}D_{i}^{k-1}.
\end{align*}}

On the other hand,

{\footnotesize\begin{align*}
    \lambda_{ij}^{m}\partial_{i}(D_{j}^{m})&=\lambda_{ij}^{m}D_{j}^{m}D_{i}-\lambda_{ij}^{m}D_{i}D_{j}^{m}\\
    &=(\lambda_{ij}^{m}-\lambda_{ji}^{m})D_{j}^{m}D_{i}-\sum_{k=1}^{m}  \biggl[Q_{k}^{m}\ x_{j}^{m-k+1}\ D_{j}D_{i}^{k-1}-P_{k}^{m}\ x_{j}^{m-k}x_{i}\ D_{i}^{k}\biggr],
\end{align*}}

which implies that 

\begin{align*}
     \lambda_{ij}^{m}\partial_{i}(D_{j}^{m}D_{i}^{n})&=\lambda_{ij}^{m}\partial_{i}(D_{j}^{m})D_{i}^{n}\\
     &=(\lambda_{ij}^{m}-\lambda_{ji}^{m})D_{j}^{m}D_{i}^{n+1}\\&-\sum_{k=1}^{m}  \biggl[Q_{k}^{m}\ x_{j}^{m-k+1}\ D_{j}D_{i}^{n+k-1}-P_{k}^{m}\ x_{j}^{m-k}x_{i}\ D_{i}^{k+n}\biggr].
\end{align*}
\end{remark}

Diffusion algebras of $n$ generators are constructed in such a way that the subalgebras of three generators are also. As we can see in \cite{IPR01} and \cite{PT02}, diffusion algebras type 1 of three generators can be classified into 4 families, $A,B,C$, and $D$, and these in turn are divided into classes as shown below.

\begin{theorem}{\rm (\cite{PT02}, p. 3270)}\label{clases}
If $\mathcal{D}$ is a diffusion algebra type 1 generated by $D_i,D_j$ and $D_k$ with $i<j<k$, then $\mathcal{D}$ belongs to some of the following classes of diffusion algebras type 1:
\begin{enumerate}
    \item Class $A_{I}$: $\lambda_{ij}=\lambda_{ji}=\lambda_{ik}=\lambda_{ki}=\lambda_{jk}=\lambda_{kj}\not=0$, $x_i,x_j,x_k\in\mathbb{K}\setminus \{0\}$.
    \item Class $A_{II}$: $\lambda_{\alpha\beta}=g_{\alpha}-g_{\beta}$ for all $\alpha,\beta\in\{i,j,k\}$ with $\alpha<\beta$, where $g_i=\lambda_{ij}+\lambda_{jk}-\lambda_{ik}$, $g_j=\lambda_{jk}-\lambda_{ik}$ and $g_k=-\lambda_{ik}$; $\lambda_{\alpha \beta}\not=0$ if $\alpha<\beta$, and $\lambda_{\alpha \beta}=0$ if $\alpha>\beta$; $x_i,x_j,x_k\in\mathbb{K}\setminus \{0\}$.
    \item Class $B_{I}$: $\lambda_{ij}=\lambda_{jk}\not=0$, $\lambda_{ji}=\lambda_{kj}$, $\lambda_{\alpha\beta}-\lambda_{\beta\alpha}=\Lambda$, for all $\alpha,\beta\in\{i,j,k\}$ with $\alpha<\beta$ and for a fixed $\Lambda\in\mathbb{K}$,  $x_i,x_k\in\mathbb{K}\setminus \{0\}$, $x_j=0$.
    \item Class $B_{II}$: $\lambda_{kj}=\lambda_{ji}=0$, $\lambda_{ij},\lambda_{ik},\lambda_{jk}\in\mathbb{K}\setminus\{0\}$, $x_j=0$.
    \item Class $B_{III}$: $\lambda_{ki}=\lambda_{kj}=0$, $\lambda_{ij}\not=0$, $\lambda_{ij}-\lambda_{ji}=\lambda_{ik}-\lambda_{jk}$, $\lambda_{ik}\not\in\{0,\lambda_{ij}-\lambda_{ji}\}$, $x_{k}=0$.
    \item Class $B_{IV}$: $\lambda_{ji}=\lambda_{ki}=0$, $x_i=0$, $\lambda_{ik}-\lambda_{ij}=\lambda_{jk}-\lambda_{kj}$, $\lambda_{ik}\not\in \{0,\lambda_{ik}-\lambda_{ij}\}$.
    \item Class $C_{I}$: $\lambda_{ij}-\lambda_{ji}=\lambda_{ik}-\lambda_{ki}$, $\lambda_{ij},\lambda_{ik},\lambda_{jk}\in\mathbb{K}\setminus\{0\}$ $x_j=x_k=0$.
    \item Class $C_{II}$: $x_j=x_k=0$, $\lambda_{kj}=0$, $\lambda_{ij},\lambda_{jk},\lambda_{ik}\in\mathbb{K}\setminus\{0\}$. 
    \item Class $D$: $x_i=x_j=x_k=0$, $\lambda_{\alpha\beta}\not=0$ for all $\alpha,\beta\in\{i,j,k\}$ with $\alpha<\beta$. 
    \end{enumerate}
\end{theorem}
 
\begin{remark}
From the results above, we conclude that diffusion algebras of type 1 for which $\lambda_{ji}\not=0$, with $i<j$, are skew polynomial algebras. In this way, the algebras presented in Theorem \ref{clases} must be also distributed in the classification of skew polynomial algebras of Theorem \ref{quasipolynomial}. More precisely,  diffusion algebras of class $C_{I}$ and $D$ are skew polynomial algebras of type \textit{2$(e)$} and \textit{1}, respectively, but the cases of diffusion algebras of classes $A_{I}$ and $B_{I}$ not belong to any of the list of Theorem \ref{quasipolynomial}. This mean that to consider the classes $A_{I}$ and $B_{I}$ as skew polynomial algebras, we need first make an identification (by establishing an automorphism of algebras) with, apparently, \textit{2$(a)$, 2$(b)$, 2$(e)$ 5$(a)$, 5$(d)$} or \textit{5$(e)$}. The rest of 3-degree diffusion algebras of type 1 are not skew polynomial algebras, at least not if we consider the same generators, by the fact that $\lambda_{ji}$'s cannot be equals to zero.
\end{remark}

We finish this section with an analysis of the imposibility of orthogonal pairs of derivations in diffusion algebras type 2. Let $\text{Aut}_L(\mathcal{D})$ be the set of automorphisms of $\mathbb{K}$-algebras of $\mathcal{D}$, where for an element $\sigma\in \text{Aut}_L(\mathcal{D})$ the relation  $\sigma(\{D_1,D_2,x_1,x_2\})\subseteq \mathbb{K}D_1+\mathbb{K}D_2+\mathbb{K}x_1+\mathbb{K}x_2+\mathbb{K}$ holds. Consider $\sigma\in \text{Aut}_L(\mathcal{D})$ and $A_\alpha,B_\alpha,S_\alpha,H_\alpha\in \mathbb{K}$, for $\alpha\in\{D1,D2,x1,x2,k\}$. The coefficients of $\sigma(D_1)$, $\sigma(D_2)$, $\sigma(x_1)$ and $\sigma(x_2)$ are given by
\begin{align}
\label{aut}
    \sigma(D_1)&=A_{D1}D_1+A_{D2}D_2+A_{x1}x_1+A_{x2}x_2+A_k,\\
    \sigma(D_2)&=B_{D1}D_1+B_{D2}D_2+B_{x1}x_1+B_{x2}x_2+B_k,\nonumber\\
    \sigma(x_1)&=S_{D1}D_1+S_{D2}D_2+S_{x1}x_1+S_{x2}x_2+S_k,\nonumber\\
    \sigma(x_2)&=H_{D1}D_1+H_{D2}D_2+H_{x1}x_1+H_{x2}x_2+H_k,\nonumber
\end{align}

which motivates the following result.

\begin{proposition}\label{propdet}
If $\sigma:\mathcal{D}\rightarrow\mathcal{D}$ is an automorphism defined as in (\ref{aut}), then ${\rm det}(A)\not=0$, where
\begin{equation}\label{laseñorasepegoenlacabeza}
    A=\left(\begin{tabular}{cccc}
        $A_{D1}$ & $B_{D1}$ & $S_{D1}$ & $H_{D1}$ \\
        $A_{D2}$ & $B_{D2}$ & $S_{D2}$ & $H_{D2}$\\
        $A_{x1}$ & $B_{x1}$ & $S_{x1}$& $H_{x1}$\\
        $A_{x2}$ & $B_{x2}$ & $S_{x2}$& $H_{x2}$
     \end{tabular}\right).
\end{equation}\end{proposition}

\begin{proof}
Suppose that $\text{det}(A)=0$. Then, the set $\{v_1,v_2,v_3,v_4\}$ of columns of $A$\footnote{If we say that $D_{1}$ is the first generator $a_1$, $D_{2}$ is $a_2$, $x_{1}$ is  $a_3$ and $x_{2}$ is $a_4$, we mean that $A_i$ is the vector of coefficients of degree one of $\sigma(a_i)$.} is linear dependent, i.e., there exist elements $\alpha_i\in\mathbb{K}$ with $i\in\{1,2,3,4\}$, such that $v_{\beta(1)}=\alpha_2v_{\beta(2)}+\alpha_3v_{\beta(3)}+\alpha_4v_{\beta(4)}$, where $\beta$ is a permutation of $\{1,2,3,4\}$. Then, $ \sigma(v_{\beta(1)}+\epsilon)=\alpha_2\sigma(v_{\beta(2)})+\alpha_3\sigma(v_{\beta(3)})+\alpha_4\sigma(v_{\beta(4)})$, where $\epsilon=\alpha_2 k_{\beta(2)}+\alpha_3k_{\beta(3)}+\alpha_4k_{\beta(4)}-k_{\beta(1)}$, with $k_1=A_k$, $k_2=B_k$, $k_3=S_k$ and $k_4=H_k$. Since $v_{\beta(1)}+\epsilon\not=\alpha_2v_{\beta(2)}+\alpha_3v_{\beta(3)}+\alpha_4v_{\beta(4)}$, we obtain that $\sigma$ is not injective, which is a contradiction. Thus, $\text{det}(A)$ is different from zero.
\end{proof}

For the next proposition, consider the following terms that can be obtained using the expression (\ref{aut}):
{\tiny
\begin{align*}
        &\lambda_{12}\sigma(D_{1})\sigma(D_{2})\\&=\lambda_{12}(A_{D1}D_{1}+A_{D2}D_{2}+A_{x1}x_{1}+A_{x2}x_{2}+A_k)(B_{D1}D_{1}+B_{D2}D_{2}+B_{x1}x_{1}+B_{x2}x_{2}+B_k)\\
        &=\lambda_{12}A_{D1}B_{D1}D_{1}^{2}+A_{D1}B_{D2}(\lambda_{21}D_{2}D_{1}+x_{2}D_{1}-x_{1}D_{2})+\lambda_{12}A_{D1}B_{x1}D_{1}x_{1}+\lambda_{12}A_{D1}B_{x2}D_{1}x_{2}\\
        &+\lambda_{12}A_{D1}B_{k}D_{1}+\lambda_{12}A_{D2}B_{D1}D_{2}D_{1}+\lambda_{12}A_{D2}B_{D2}D_{2}^{2}+\lambda_{12}A_{D2}B_{x1}D_{2}x_{1}+\lambda_{12}A_{D2}B_{x2}D_{2}x_{2}\\
        &+\lambda_{12}A_{D2}B_{k}D_{2}+\lambda_{12}A_{x1}B_{D1}x_{1}D_{1}+\lambda_{12}A_{x1}B_{D2}x_{1}D_{2}+\lambda_{12}A_{x1}B_{x1}x_{1}^{2}+\lambda_{12}A_{x1}B_{x2}x_{1}x_{2}\\
        &+\lambda_{12}A_{x1}B_{k}x_{1}+\lambda_{12}A_{x2}B_{D1}x_{2}D_{1}+\lambda_{12}A_{x2}B_{D2}x_{2}D_{2}+\lambda_{12}A_{x2}B_{x1}x_{2}x_{1}+\lambda_{12}A_{x2}B_{x2}x_{2}^{2}\\
        &+\lambda_{12}A_{x2}B_{k}x_{2}+\lambda_{12}A_{k}B_{D1}D_{1}+\lambda_{12}A_{k}B_{D2}D_{2}+\lambda_{12}A_{k}B_{x1}x_{1}+\lambda_{12}A_{k}B_{x2}x_{2}+\lambda_{12}A_{k}B_{k},\\\\
        &\lambda_{21}\sigma(D_{2})\sigma(D_{1})\\&=\lambda_{12}(B_{D1}D_{1}+B_{D2}D_{2}+B_{x1}x_{1}+B_{x2}x_{2}+B_k)(A_{D1}D_{1}+A_{D2}D_{2}+A_{x1}x_{1}+A_{x2}x_{2}+A_k)\\
        &=\lambda_{21}B_{D1}A_{D1}D_{1}^{2}+\lambda_{21}B_{D1}A_{D2}\lambda_{12}^{-1}(\lambda_{21}D_{2}D_{1}+x_{2}D_{1}-x_{1}D_{2})+\lambda_{21}B_{D1}A_{x1}D_{1}x_{1}+\lambda_{21}B_{D1}A_{x2}D_{1}x_{2}\\
        &+\lambda_{21}B_{D1}A_{k}D_{1}+\lambda_{21}B_{D2}A_{D1}D_{2}D_{1}+\lambda_{21}B_{D2}A_{D2}D_{2}^{2}+\lambda_{21}B_{D2}A_{x1}D_{2}x_{1}+\lambda_{21}B_{D2}A_{x2}D_{2}x_{2}\\
        &+\lambda_{21}B_{D2}A_{k}D_{2}+\lambda_{21}B_{x1}A_{D1}x_{1}D_{1}+\lambda_{21}B_{x1}A_{D2}x_{1}D_{2}+\lambda_{21}B_{x1}A_{x1}x_{1}^{2}+\lambda_{21}B_{x1}A_{x2}x_{1}x_{2}\\
        &+\lambda_{21}B_{x1}A_{k}x_{1}+\lambda_{21}B_{x2}A_{D1}x_{2}D_{1}+\lambda_{21}B_{x2}A_{D2}x_{2}D_{2}+\lambda_{21}B_{x2}A_{x1}x_{2}x_{1}+\lambda_{21}B_{x2}A_{x2}x_{2}^{2}\\
        &+\lambda_{21}B_{x2}A_{k}x_{2}+\lambda_{21}B_{k}A_{D1}D_{1}+\lambda_{21}B_{k}A_{D2}D_{2}+\lambda_{21}B_{k}A_{x1}x_{1}+\lambda_{21}B_{k}A_{x2}x_{2}+\lambda_{21}B_{k}A_{k},\\\\
        &\sigma(x_{2})\sigma(D_{1})\\&=(H_{D1}D_{1}+H_{D2}D_{2}+H_{x1}x_{1}+H_{x2}x_{2}+H_k)(A_{D1}D_{1}+A_{D2}D_{2}+A_{x1}x_{1}+A_{x2}x_{2}+A_k)\\
        &=H_{D1}A_{D1}D_{1}^{2}+H_{D1}A_{D2}\lambda_{12}^{-1}(\lambda_{21}D_{2}D_{1}+x_{2}D_{1}-x_{1}D_{2})+H_{D1}A_{x1}D_{1}x_{1}+H_{D1}A_{x2}D_{1}x_{2}\\
        &+H_{D1}A_{k}D_{1}+H_{D2}A_{D1}D_{2}D_{1}+H_{D2}A_{D2}D_{2}^{2}+H_{D2}A_{x1}D_{2}x_{1}+H_{D2}A_{x2}D_{2}x_{2}\\
        &+H_{D2}A_{k}D_{2}+H_{x1}A_{D1}x_{1}D_{1}+H_{x1}A_{D2}x_{1}D_{2}+H_{x1}A_{x1}x_{1}^{2}+H_{x1}A_{x2}x_{1}x_{2}\\
        &+H_{x1}A_{k}x_{1}+H_{x2}A_{D1}x_{2}D_{1}+H_{x2}A_{D2}x_{2}D_{2}+H_{x2}A_{x1}x_{2}x_{1}+H_{x2}A_{x2}x_{2}^{2}\\
        &+H_{x2}A_{k}x_{2}+H_{k}A_{D1}D_{1}+H_{k}A_{D2}D_{2}+H_{k}A_{x1}x_{1}+H_{k}A_{x2}x_{2}+H_{k}A_{k},\\\\
        &\sigma(x_{1})\sigma(D_{2})\\&=(S_{D1}D_{1}+S_{D2}D_{2}+S_{x1}x_{1}+S_{x2}x_{2}+S_k)(B_{D1}D_{1}+B_{D2}D_{2}+B_{x1}x_{1}+B_{x2}x_{2}+B_k)\\
        &=S_{D1}B_{D1}D_{1}^{2}+S_{D1}B_{D2}\lambda_{12}^{-1}(\lambda_{21}D_{2}D_{1}+x_{2}D_{1}-x_{1}D_{2})+S_{D1}B_{x1}D_{1}x_{1}+S_{D1}B_{x2}D_{1}x_{2}+S_{D1}B_{k}D_{1}\\
        &+S_{D2}B_{D1}D_{2}D_{1}+S_{D2}B_{D2}D_{2}^{2}+S_{D2}B_{x1}D_{2}x_{1}+S_{D2}B_{x2}D_{2}x_{2}+S_{D2}B_{k}D_{2}\\
        &+S_{x1}B_{D1}x_{1}D_{1}+S_{x1}B_{D2}x_{1}D_{2}+S_{x1}B_{x1}x_{1}^{2}+S_{x1}B_{x2}x_{1}x_{2}+S_{x1}B_{k}x_{1}\\
        &+S_{x2}B_{D1}x_{2}D_{1}+S_{x2}B_{D2}x_{2}D_{2}+S_{x2}B_{x1}x_{2}x_{1}+S_{x2}B_{x2}x_{2}^{2}+S_{x2}B_{k}x_{2}\\
        &+S_{k}B_{D1}D_{1}+S_{k}B_{D2}D_{2}+S_{k}B_{x1}x_{1}+S_{k}B_{x2}x_{2}+S_{k}B_{k}.
    \end{align*}}

\begin{proposition}\label{determinante}
If $\sigma:\mathcal{D}\rightarrow\mathcal{D}$ is an automorphism defined as in {\rm (\ref{aut})}, then $A_k=B_k=S_k=H_k=0$.
\end{proposition}
\begin{proof} As we must guarantee that (\ref{relations}) is respected by $\sigma$, we need to verify that the following relation holds:
\begin{equation}\label{automorphismrule2}
      \lambda_{12}\sigma(D_{1})\sigma(D_{2})-\lambda_{21}\sigma(D_{2})\sigma(D_{1})-\sigma(x_{2})\sigma(D_{1})+\sigma(x_{1})\sigma(D_{2})=0.
\end{equation}

Therefore, we should have that the coefficients of the terms of degree one $D_{1},D_{2}, x_{1}$ and $x_{2}$ must be zero (this due to the fact that $\mathcal{D}$ is a quadratic algebra). With this in mind, we vanish the coefficients of the elements $D_{1},D_{2},x_{1}$ and $x_{2}$ in expression (\ref{automorphismrule2}), whence we obtain the following equations:
{\tiny \begin{align*}
 D_{1}: \ \ \    
 &\lambda_{12}A_{D1}B_{k} +\lambda_{12}A_{k}B_{D1}-\lambda_{21}B_{D1}A_{k}-\lambda_{21}B_{k}A_{D1}-H_{D1}A_{k}-H_{k}A_{D1}+S_{D1}B_{k}+S_{k}B_{D1}=0\\
 D_{2}: \ \ \   
 &\lambda_{12}A_{D2}B_{k}+\lambda_{12}A_{k}B_{D2}-\lambda_{21}B_{D2}A_{k}-\lambda_{21}B_{k}A_{D2}-H_{D2}A_{k}-H_{k}A_{D2}+S_{D2}B_{k}+S_{k}B_{D2}=0\\
 x_{1}:\ \ \ & \  \lambda_{12}A_{x1}B_{k}+\lambda_{12}A_{k}B_{x1}-\lambda_{21}B_{x1}A_{k}-\lambda_{21}B_{k}A_{x1}-H_{x1}A_{k}-H_{k}A_{x1}+S_{x1}B_{k}+S_{k}B_{x1}=0\\
 x_{2}:\ \ \ & \   \lambda_{12}A_{x2}B_{k}+\lambda_{12}A_{k}B_{x2}-\lambda_{21}B_{x2}A_{k}-\lambda_{21}B_{k}A_{x2}-H_{x2}A_{k}-H_{k}A_{x2}+S_{x2}B_{k}+S_{k}B_{x2}=0\\
    k:\ \ \ &\  \lambda_{12}A_{k}B_{k}-\lambda_{21}B_{k}A_{k}-H_{k}A_{k}+S_{k}B_{k}=0.
\end{align*}}

\noindent From the algebraic properties in $\mathbb{K}$, we obtain the following equalities:
{\scriptsize \begin{align*}
    D_{1}&: \ \ \ \ \ \ [(\lambda_{12}-\lambda_{21})B_{D1}-H_{D1}]A_{k}+[(\lambda_{12}-\lambda_{21})A_{D1}+S_{D1}]B_{k}+B_{D1}S_{k}-A_{D1}H_{k}=0,\\    D_{2}&: \ \ \ \ \ \ [(\lambda_{12}-\lambda_{21})B_{D2}-H_{D2}]A_{k}+[(\lambda_{12}-\lambda_{21})A_{D2}+S_{D2}]B_{k}+B_{D2}S_{k}-A_{D2}H_{k}=0,\\
    x_{1}&: \ \ \ \ \ \ [(\lambda_{12}-\lambda_{21})B_{x1}-H_{x1}]A_{k}+[(\lambda_{12}-\lambda_{21})A_{x1}+S_{x1}]B_{k}+B_{x1}S_{k}-A_{x1}H_{k}=0,\\
    x_{2}&: \ \ \ \ \ \ [(\lambda_{12}-\lambda_{21})B_{x2}-H_{x2}]A_{k}+[(\lambda_{12}-\lambda_{21})A_{x2}+S_{x2}]B_{k}+B_{x2}S_{k}-A_{x2}H_{k}=0,\\
    k&:\ \ \ \ \ \ (\lambda_{12}-\lambda_{21})A_kB_k+S_kB_k-H_kA_k=0,
\end{align*}}

Equivalently, the equations obtained in the generators $D_{1}, D_{2}, x_{1}$ and $x_{2}$ can be expressed as the linear system $\Gamma \overline{x}=0$, where
{\tiny \begin{equation*}
    \Gamma=\left(\begin{tabular}{cccc}
    $[(\lambda_{12}-\lambda_{21})B_{D1}-H_{D1}]$&$[(\lambda_{12}-\lambda_{21})A_{D1}+S_{D1}]$&$B_{D1}$&$-A_{D1}$\\
          $[(\lambda_{12}-\lambda_{21})B_{D2}-H_{D2}]$&$[(\lambda_{12}-\lambda_{21})A_{D2}+S_{D2}]$&$B_{D2}$&$-A_{D2}$\\
         $[(\lambda_{12}-\lambda_{21})B_{x1}-H_{x1}]$&$[(\lambda_{12}-\lambda_{21})A_{x1}+S_{x1}]$&$B_{x1}$&$-A_{x1}$\\
         $[(\lambda_{12}-\lambda_{21})B_{x2}-H_{x2}]$&$[(\lambda_{12}-\lambda_{21})A_{x2}+S_{x2}]$&$B_{x2}$&$-A_{x2}$
     \end{tabular}\right),\ \ 
     \overline{x}=\left(\begin{tabular}{c}
          $A_k$ \\ $B_k$ \\ $S_k$ \\ $H_k$
     \end{tabular}\right).
\end{equation*}}
By the properties of the determinant function, we get $\text{det}(\Gamma)=\text{det}(A)$, see (\ref{laseñorasepegoenlacabeza}). In this way, Proposition \ref{propdet} implies that $\text{det}(A)\not= 0$, whence the system $\Gamma \overline{x}=0$ has the unique solution $\overline{x}=0$, that is, $A_k=B_k=S_k=H_k=0$.
\end{proof}

Now, for $\sigma:\mathcal{D}\rightarrow\mathcal{D}$ defined by (\ref{aut}), and for a $\sigma$-derivation $\partial:\mathcal{D}\rightarrow\mathcal{D}$ defined on basic elements as follows
\begin{equation}\label{deri}
    \partial(D_1)=a_{D1}D_1+a_{D2}D_2+a_{x1}x_1+a_{x2}x_2+a_k,
\end{equation}
\begin{equation*}
    \partial(D_2)=b_{D1}D_1+b_{D2}D_2+b_{x1}x_1+b_{x2}x_2+b_k,
\end{equation*}
\begin{equation*}
    \partial(x_1)=c_{D1}D_1+c_{D2}D_2+c_{x1}x_1+c_{x2}x_2+c_k,
\end{equation*}
\begin{equation*}
    \partial(x_2)=d_{D1}D_1+d_{D2}D_2+d_{x1}x_1+d_{x2}x_2+d_k,
\end{equation*}

we can check that

{\tiny
\begin{align*}
    &\lambda_{12}\partial(D_{1})\sigma(D_{2})\\&=\lambda_{12}(a_{D1}D_{1}+a_{D2}D_{2}+a_{x1}x_{1}+a_{x2}x_{2}+a_k)(B_{D1}D_{1}+B_{D2}D_{2}+B_{x1}x_{1}+B_{x2}x_{2}+B_k)\\
    &=\lambda_{12}a_{D1}B_{D1}D_{1}^2+a_{D1}B_{D2}(\lambda_{21}D_{2}D_{1}+x_{2}D_{1}-x_{1}D_{2})+\lambda_{12}a_{D1}B_{x1}D_{1}x_{1}+\lambda_{12}a_{D1}B_{x2}D_{1}x_{2}\\
    &+\lambda_{12}a_{D1}B_{k}D_{1}+\lambda_{12}a_{D2}B_{D1}D_{2}D_{1}+\lambda_{12}a_{D2}B_{D2}D_{2}^2+\lambda_{12}a_{D2}B_{x1}D_{2}x_{1}+\lambda_{12}a_{D2}B_{x2}D_{2}x_{2}\\    
    &+\lambda_{12}a_{D2}B_{k}D_{2}+\lambda_{12}a_{x1}B_{D1}x_{1}D_{1}+\lambda_{12}a_{x1}B_{D2}x_{1}D_{2}+\lambda_{12}a_{x1}B_{x1}x_{1}^2+\lambda_{12}a_{x1}B_{x2}x_{1}x_{2}\\
    &+\lambda_{12}a_{x1}B_{k}x_{1}+\lambda_{12}a_{x2}B_{D1}x_{2}D_{1}+\lambda_{12}a_{x2}B_{D2}x_{2}D_{2}+\lambda_{12}a_{x2}B_{x1}x_{2}x_{1}+\lambda_{12}a_{x2}B_{x2}x_{2}^2\\
    &+\lambda_{12}a_{x2}B_{k}x_{2}+\lambda_{12}a_{k}B_{D1}D_{1}+\lambda_{12}a_{k}B_{D2}D_{2}+\lambda_{12}a_{k}B_{x1}x_{1}+\lambda_{12}a_{k}B_{x2}x_{2}+\lambda_{12}a_{k}B_{k},\\
     &\lambda_{12}D_{1}\partial(D_{2})\\&=\lambda_{12}D_{1}(b_{D1}D_{1}+b_{D2}D_{2}+b_{x1}x_{1}+b_{x2}x_{2}+b_k)\\
    &=\lambda_{12}b_{D1}D_{1}^2+b_{D2}(\lambda_{21}D_{2}D_{1}+x_{2}D_{1}-x_{1}D_{2})+\lambda_{12}b_{x1}D_{1}x_{1}+\lambda_{12}b_{x2}D_{1}x_{2}+\lambda_{12}b_{k}D_{1},
    \end{align*}}
    
    {\tiny
\begin{align*}
     &\lambda_{21}\partial(D_{2})\sigma(D_{1})\\&=\lambda_{21}(b_{D1}b_{1}+b_{D2}D_{2}+b_{x1}x_{1}+b_{x2}x_{2}+b_k)(A_{D1}D_{1}+A_{D2}D_{2}+A_{x1}x_{1}+A_{x2}x_{2}+A_k)\\
    &=\lambda_{21}b_{D1}A_{D1}D_{1}^2+\lambda_{21}b_{D1}A_{D2}\lambda_{12}^{-1}(\lambda_{21}D_{2}D_{1}+x_{2}D_{1}-x_{1}D_{2})+\lambda_{21}b_{D1}A_{x1}D_{1}x_{1}+\lambda_{21}b_{D1}A_{x2}D_{1}x_{2}\\
    &+\lambda_{21}b_{D1}A_{k}D_{1}+\lambda_{21}b_{D2}A_{D1}D_{2}D_{1}+\lambda_{21}b_{D2}A_{D2}D_{2}^2+\lambda_{21}b_{D2}A_{x1}D_{2}x_{1}+\lambda_{21}b_{D2}A_{x2}D_{2}x_{2}\\    
    &+\lambda_{21}b_{D2}A_{k}D_{2}+\lambda_{21}b_{x1}A_{D1}x_{1}D_{1}+\lambda_{21}b_{x1}A_{D2}x_{1}D_{2}+\lambda_{21}b_{x1}A_{x1}x_{1}^2+\lambda_{21}b_{x1}A_{x2}x_{1}x_{2}\\
    &+\lambda_{21}b_{x1}A_{k}x_{1}+\lambda_{21}b_{x2}A_{D1}x_{2}D_{1}+\lambda_{21}b_{x2}A_{D2}x_{2}D_{2}+\lambda_{21}b_{x2}A_{x1}x_{2}x_{1}+\lambda_{21}b_{x2}A_{x2}x_{2}^2\\
    &+\lambda_{21}b_{x2}A_{k}x_{2}+\lambda_{21}b_{k}A_{D1}D_{1}+\lambda_{21}b_{k}A_{D2}D_{2}+\lambda_{21}b_{k}A_{x1}x_{1}+\lambda_{21}b_{k}A_{x2}x_{2}+\lambda_{21}b_{k}A_{k},\\\\
    &\lambda_{21}D_{2}\partial(D_{1})\\&=\lambda_{21}D_{2}(a_{D1}D_{1}+a_{D2}D_{2}+a_{x1}x_{1}+a_{x2}x_{2}+a_k)\\
    &=\lambda_{21}a_{D1}D_{2}D_{1}+\lambda_{21}a_{D2}D_{2}^2+\lambda_{21}a_{x1}D_{2}x_{1}+\lambda_{21}a_{x2}D_{2}x_{2}+\lambda_{21}a_{k}D_{2},\\\\
    &\partial(x_{2})\sigma(D_{1})\\&=(d_{D1}D_{1}+d_{D2}D_{2}+d_{x1}x_{1}+d_{x2}x_{2}+d_k)(A_{D1}D_{1}+A_{D2}D_{2}+A_{x1}x_{1}+A_{x2}x_{2}+A_k)\\
    &=d_{D1}A_{D1}D_{1}^2+d_{D1}A_{D2}\lambda_{12}^{-1}(\lambda_{21}D_{2}D_{1}+x_{2}D_{1}-x_{1}D_{2})+d_{D1}A_{x1}D_{1}x_{1}+d_{D1}A_{x2}D_{1}x_{2}+d_{D1}A_{k}D_{1}\\
    &+d_{D2}A_{D1}D_{2}D_{1}+d_{D2}A_{D2}D_{2}^2+d_{D2}A_{x1}D_{2}x_{1}+d_{D2}A_{x2}D_{2}x_{2}+d_{D2}A_{k}D_{2}\\    
    &+d_{x1}A_{D1}x_{1}D_{1}+d_{x1}A_{D2}x_{1}D_{2}+d_{x1}A_{x1}x_{1}^2+d_{x1}A_{x2}x_{1}x_{2}+d_{x1}A_{k}x_{1}\\
    &+B_{x2}A_{D1}x_{2}D_{1}+d_{x2}A_{D2}x_{2}D_{2}+d_{x2}A_{x1}x_{2}x_{1}+d_{x2}A_{x2}x_{2}^2+d_{x2}A_{k}x_{2}\\
    &+d_{k}A_{D1}D_{1}+d_{k}A_{D2}D_{2}+d_{k}A_{x1}x_{1}+d_{k}A_{x2}x_{2}+d_{k}A_{k},\\\\
    &x_{2}\partial(D_{1})\\&=x_{2}(a_{D1}D_{1}+a_{D2}D_{2}+a_{x1}x_{1}+a_{x2}x_{2}+a_k)\\
    &=a_{D1}x_{2}D_{1}+a_{D2}x_{2}D_{2}+a_{x1}x_{2}x_{1}+a_{x2}x_{2}^2+a_{k}x_{2},\\\\
    &\partial(x_{1})\sigma(D_{2})\\&=(c_{D1}D_{1}+c_{D2}D_{2}+c_{x1}x_{1}+c_{x2}x_{2}+c_k)(B_{D1}D_{1}+B_{D2}D_{2}+B_{x1}x_{1}+B_{x2}x_{2}+B_k)\\
    &=c_{D1}B_{D1}D_{1}^2+c_{D1}B_{D2}(\lambda_{21}D_{2}D_{1}+x_{2}D_{1}-x_{1}D_{2})+c_{D1}B_{x1}D_{1}x_{1}+c_{D1}B_{x2}D_{1}x_{2}+c_{D1}B_{k}D_{1}\\
    &+c_{D2}B_{D1}D_{2}D_{1}+c_{D2}B_{D2}D_{2}^2+c_{D2}B_{x1}D_{2}x_{1}+c_{D2}B_{x2}D_{2}x_{2}+c_{D2}B_{k}D_{2}\\    
    &+c_{x1}B_{D1}x_{1}D_{1}+c_{x1}B_{D2}x_{1}D_{2}+c_{x1}B_{x1}x_{1}^2+c_{x1}B_{x2}x_{1}x_{2}+c_{x1}B_{k}x_{1}\\
    &+c_{x2}B_{D1}x_{2}D_{1}+c_{x2}B_{D2}x_{2}D_{2}+c_{x2}B_{x1}x_{2}x_{1}+c_{x2}B_{x2}x_{2}^2+c_{x2}B_{k}x_{2}\\
    &+c_{k}B_{D1}D_{1}+c_{k}B_{D2}D_{2}+c_{k}B_{x1}x_{1}+c_{k}B_{x2}x_{2}+c_{k}B_{k},\\\\
    &x_{1}\partial(D_{2})\\&=x_{1}(b_{D1}D_{1}+b_{D2}D_{2}+b_{x1}x_{1}+b_{x2}x_{2}+b_k)\\
    &=b_{D1}x_{1}D_{1}+b_{D2}x_{1}D_{2}+b_{x1}x_{1}^2+b_{x2}x_{1}x_{2}+b_{k}x_{1}.
\end{align*}}

With these previous terms, we obtain the following proposition.

\begin{proposition}
\label{teon2}
If $\sigma:\mathcal{D}\rightarrow\mathcal{D}$ is an automorphism defined as in {\rm(\ref{aut})} such that ${\rm span}_{\mathbb{K}}(S,H)={\rm span}_{\mathbb{K}}(L_1,L_2)$, and $\partial:\mathcal{D}\rightarrow\mathcal{D}$ is a $\sigma$-derivation as in {\rm(\ref{deri})}, then the elements $\partial(D_{1}),\ \partial(D_{2}),\ \partial(x_{1})$ and $\partial(x_{2})$ have no zero degree terms, where 
{\small\begin{equation*}
    L:= \left(\begin{tabular}{cccc}
 $0$ & $\lambda_{12}$ &$A_{D1}$&$B_{D1}$\\
$-\lambda_{21}$ & $0$ & $A_{D2}$ & $B_{D2}$\\
$0$&$1$&$A_{x1}$&$B_{x1}$\\
$1$&$0$&$A_{x2}$&$B_{x2}$
    \end{tabular}\right),
\end{equation*}

\begin{equation*}
    L_1: =\left(\begin{tabular}{c}
 $0$ \\
$-\lambda_{21}$ \\
$0$\\
$1$
    \end{tabular}\right),\ \ 
     L_2: =\left(\begin{tabular}{c}
$\lambda_{12}$ \\
 $0$ \\
$1$\\
$0$
    \end{tabular}\right),\ \ 
S: =\left(\begin{tabular}{c}
 $S_{D1}$ \\
$S_{D2}$ \\
$S_{x1}$\\
$S_{x2}$
    \end{tabular}\right),\ \ 
H: =\left(\begin{tabular}{c}
 $H_{D1}$ \\
$H_{D2}$ \\
$H_{x1}$\\
$H_{x2}$
    \end{tabular}\right).
\end{equation*}}
\end{proposition}
\begin{proof}
From the one degree terms of $\partial$ applied to the relation (\ref{relations}), we obtain the equations:
{\tiny \begin{align*}
        D_{1}:\ 
        & \lambda_{12}a_{D1}B_{k}+\lambda_{12}a_{k}B_{D1}+\lambda_{12}b_{k}-\lambda_{21}B_{D1}A_{k}-\lambda_{21}b_{k}A_{D1}-d_{D1}A_{k}-d_{k}A_{D1}+c_{D1}B_{k}+c_{k}B_{D1}=0\\
        D_{2}:\ &\lambda_{12}a_{D2}B_{k}+\lambda_{12}a_{k}B_{D2}-\lambda_{21}b_{k}A_{D2}-\lambda_{21}b_{D2}A_{k}-\lambda_{21}a_{k}-d_{D2}A_{k}-d_{k}A_{D2}+c_{D2}B_{k}+c_{k}B_{D2}=0\\
        x_{1}:\  &\lambda_{12}a_{x1}B_{k}+\lambda_{12}a_{k}B_{x1}-\lambda_{21}b_{k}A_{x1}-\lambda_{21}b_{x1}A_{k}-d_{x1}A_{k}-d_{k}A_{x1}+c_{x1}B_{k}+c_{k}B_{x1}+b_k=0 \\
        x_{2}:\  &\lambda_{12}a_{x2}B_{k}+\lambda_{12}a_{k}B_{x2}-\lambda_{21}b_{k}A_{x2}-\lambda_{21}b_{x2}A_{k}-d_{x2}A_{k}-d_{k}A_{x2}+c_{x2}B_{k}+c_{k}B_{x2}+a_k=0
\end{align*}}

Now, since $\text{det}(A)\not=0$, then Proposition  \ref{determinante} implies that $A_k=B_k=0$. In this way, the previous equations can be expressed as 
\begin{align*}
        D_{1}&:\ \ \ \lambda_{12}a_{k}B_{D1}-\lambda_{21}b_{k}A_{D1}-d_{k}A_{D1}+c_{k}B_{D1}=-\lambda_{12}b_{k},\\
        D_{2}&:\ \ \ \lambda_{12}a_{k}B_{D2}-\lambda_{21}b_{k}A_{D2}-d_{k}A_{D2}+c_{k}B_{D2}=\lambda_{21}a_{k},\\
        x_{1}&:\ \ \ \lambda_{12}a_{k}B_{x1}-\lambda_{21}b_{k}A_{x1}-d_{k}A_{x1}+c_{k}B_{x1}=-b_k,\\
        x_{2}&:\ \ \ \lambda_{12}a_{k}B_{x2}-\lambda_{21}b_{k}A_{x2}-d_{k}A_{x2}+c_{k}B_{x2}=-a_k,
\end{align*}

which is a linear system of equations in the variables $a_k,b_k,c_k$ and $d_k$, i.e., it is a system $\Theta\overline{y}=0$, where
\begin{equation*}
    \Theta=\left(\begin{tabular}{cccc}
 $\lambda_{12}B_{D1}$ & $\lambda_{12}-\lambda_{21}A_{D1}$ &$-A_{D1}$&$B_{D1}$\\
$\lambda_{12}B_{D2}-\lambda_{21}$ & $-\lambda_{21}A_{D2}$ & $-A_{D2}$ & $B_{D2}$\\
$\lambda_{12}B_{x1}$&$1-\lambda_{21}A_{x1}$&$-A_{x1}$&$B_{x1}$\\
$\lambda_{12}B_{x2}+1$&$-\lambda_{21}A_{x2}$&$-A_{x2}$&$B_{x2}$
    \end{tabular}\right),\ \ \ \overline{y}=\left(\begin{tabular}{c}
         $a_k$\\
         $b_k$ \\
          $d_k$\\
          $c_k$
    \end{tabular}\right).
\end{equation*}

It is clear that 

\begin{equation*}
\begin{split}
    \text{det}(\Theta)&=\text{det}\left(\begin{tabular}{cccc}
 $\lambda_{12}B_{D1}$ & $\lambda_{12}-\lambda_{21}A_{D1}$ &$-A_{D1}$&$B_{D1}$\\
$\lambda_{12}B_{D2}$ & $-\lambda_{21}A_{D2}$ & $-A_{D2}$ & $B_{D2}$\\
$\lambda_{12}B_{x1}$&$1-\lambda_{21}A_{x1}$&$-A_{x1}$&$B_{x1}$\\
$\lambda_{12}B_{x2}$&$-\lambda_{21}A_{x2}$&$-A_{x2}$&$B_{x2}$
    \end{tabular}\right)\\&+\text{det}\left(\begin{tabular}{cccc}
 $0$ & $\lambda_{12}-\lambda_{21}A_{D1}$ &$-A_{D1}$&$B_{D1}$\\
$-\lambda_{21}$ & $-\lambda_{21}A_{D2}$ & $-A_{D2}$ & $B_{D2}$\\
$0$&$1-\lambda_{21}A_{x1}$&$-A_{x1}$&$B_{x1}$\\
$1$&$-\lambda_{21}A_{x2}$&$-A_{x2}$&$B_{x2}$
    \end{tabular}\right)\\
    &=\text{det}\left(\begin{tabular}{cccc}
 $0$ & $\lambda_{12}-\lambda_{21}A_{D1}$ &$-A_{D1}$&$B_{D1}$\\
$-\lambda_{21}$ & $-\lambda_{21}A_{D2}$ & $-A_{D2}$ & $B_{D2}$\\
$0$&$1-\lambda_{21}A_{x1}$&$-A_{x1}$&$B_{x1}$\\
$1$&$-\lambda_{21}A_{x2}$&$-A_{x2}$&$B_{x2}$
    \end{tabular}\right)
\end{split}
\end{equation*}
\begin{equation*}
    \begin{split}
    &=\text{det}\left(\begin{tabular}{cccc}
 $0$ & $-\lambda_{21}A_{D1}$ &$-A_{D1}$&$B_{D1}$\\
$-\lambda_{21}$ & $-\lambda_{21}A_{D2}$ & $-A_{D2}$ & $B_{D2}$\\
$0$&$-\lambda_{21}A_{x1}$&$-A_{x1}$&$B_{x1}$\\
$1$&$-\lambda_{21}A_{x2}$&$-A_{x2}$&$B_{x2}$
    \end{tabular}\right)\\
    &+\text{det}\left(\begin{tabular}{cccc}
 $0$ & $\lambda_{12}$ &$-A_{D1}$&$B_{D1}$\\
$-\lambda_{21}$ & $0$ & $-A_{D2}$ & $B_{D2}$\\
$0$&$1$&$-A_{x1}$&$B_{x1}$\\
$1$&$0$&$-A_{x2}$&$B_{x2}$
    \end{tabular}\right)\\
    &=\text{det}\left(\begin{tabular}{cccc}
 $0$ & $\lambda_{12}$ &$-A_{D1}$&$B_{D1}$\\
$-\lambda_{21}$ & $0$ & $-A_{D2}$ & $B_{D2}$\\
$0$&$1$&$-A_{x1}$&$B_{x1}$\\
$1$&$0$&$-A_{x2}$&$B_{x2}$
    \end{tabular}\right)=-\text{det}(L)
\end{split}
    \end{equation*}
    
Due to the fact that $\text{span}_{\mathbb{K}}(S,H)=\text{span}_{\mathbb{K}}(L_1,L_2)$, we have that the matrices $A^T$ and $L^T$ are row equivalent, and therefore $\text{det}(L)\not=0$, because $\text{det}(A)\not=0$. Since $\text{det}(\Theta)=-\text{det}(L)$, $\text{det}(\Theta)\not=0$ which implies that the unique solution to the homogeneous linear system $\Theta\overline{y}=0$ is the trivial solution $\overline{y}=0$, i.e., $a_k=b_k=c_k=d_k=0$.
\end{proof}

From results above, it follows the next assertion.

\begin{corollary}\label{impo}
If $\sigma:\mathcal{D}\rightarrow\mathcal{D}$ is an automorphism defined as in {\rm (\ref{aut})} such that the linear components satisfy  ${\rm span}_{\mathbb{K}}(S,H)={\rm span}_{\mathbb{K}}(L_1,L_2)$ as in  {\rm Proposition \ref{teon2}}, and if $\partial:\mathcal{D}\rightarrow\mathcal{D}$ is a $\sigma$-derivation, then $\partial(\mathcal{D})\cap \mathbb{K}=\emptyset$.
\end{corollary}

By Corollary \ref{impo}, we cannot guarantee the density conditions (\cite{AB18}, p. 6), which, as we have seen, we want for the construction of the differential calculus used in Theorem \ref{Theorem3d} and Theorem \ref{mostimportant}. Therefore, when we work in diffusion algebras type 2 with 2 generators, we prefer to choose skew derivations of graded automorphism such that ${\rm span}_{\mathbb{K}}(S,H)\not={\rm span}_{\mathbb{K}}(L_1,L_2)$.


%
%

\section{Conclusions and future work}

In this article, we have studied the differential smoothness of several algebras, and we have realized that the GK dimension is a very important notion to guarantee the non-existence of the differential calculus on the algebras. It is a pending task to characterize the differential smoothness of 3-dimensional skew polynomial algebras type $5(e)$. 

Now, since the 3-dimensional skew polynomial and diffusion algebras are part of more general families of PBW noncommutative structures (see for example \cite{LFGRSV}, \cite{ReyesRodriguez2021}, and \cite{ReyesSuarez2021}) for which its GK dimension has been computed explicitly in \cite{LFGRSV}, a possible future work is to study differential smoothness in this general context. For example, in the setting of skew polynomial rings considered by Artamonov et al., \cite{Artamonovetal2016}, which are examples of the PBW algebras studied in \cite{LFGRSV}, we can see that the coefficients do not commute with the variables, as they do with the algebras studied here, so new calculations will have to be developed. Surely, the derivations of these structures following Artamonov's ideas \cite{Artamonov2015}, and the automorphisms of the objects characterized by Venegas \cite{Venegas2015} will have to considered.

%
%



\end{document}